\documentclass[12pt]{amsart}
\usepackage{amsmath}
\usepackage{amssymb}
\usepackage{graphicx}
\usepackage{pinlabel}
\usepackage[all]{xy}

\usepackage{amsmath,amssymb,amsfonts,amsthm,graphicx}

\usepackage{pinlabel}

\usepackage[usenames,dvipsnames]{color}

\usepackage[usenames,dvipsnames]{color}

\newtheorem{dummy}{dummy}[section]
\newtheorem{lemma}[dummy]{Lemma}

\newtheorem{corollary}[dummy]{Corollary}
\newtheorem{proposition}[dummy]{Proposition}

\theoremstyle{definition}

\newtheorem{remark}[dummy]{Remark}

\newtheorem{question}[dummy]{Question}

\numberwithin{equation}{section}

\newcommand{\R}{\mathbb {R}}

\newcommand{\Z}{\mathbb {Z}}
\newcommand{\F}{\mathbb {F}}

\newcommand{\e}{\epsilon}

\newcommand{\Reeb}{\mathcal{R}}

\newcommand{\Char}{\mbox{char}\,}
\newcommand{\Loc}{\mathit{Aug}(L, \mathbb{F})}

\begin{document}

\title[Non-fillable augmentations of twist knots]{Non-fillable augmentations of twist knots}

\author{Honghao Gao and Dan Rutherford}

\begin{abstract}  We establish new examples of augmentations of Legendrian twist knots that cannot be induced by orientable Lagrangian fillings.  To do so, we use a version of the Seidel-Ekholm-Dimitroglou Rizell isomorphism with local coefficients to  show that any Lagrangian filling point in the augmentation variety of a Legendrian knot must lie in the injective image of an algebraic torus with dimension equal to the first Betti number of the filling.  This is a Floer-theoretic version of a result from microlocal sheaf theory.  For the augmentations in question, we show that no such algebraic torus can exist.
\end{abstract}

\maketitle

\section{Introduction}  Let $\Lambda \subset \R^3$ be a Legendrian knot in $\R^3$ with its standard contact structure.  A challenging geometric problem is to classify the possible Lagrangian fillings of $\Lambda$.  An exact Lagrangian filling $L$ of $\Lambda$ in the symplectization $\mathit{Symp}(\R^3) = \R\times \R^3$ with vanishing Maslov number, $m(L) =0$, induces an algebraic object: an augmentation of the Legendrian contact DGA (differential graded algebra) of $\Lambda$, i.e. a DGA homomorphism 
\[
\epsilon:\mathcal{A}(\Lambda) \rightarrow \mathbb{F}_2.
\]
See \cite{EHK}.  Augmentations to more general fields, $\mathbb{F}$, can be induced by equipping $L$ with some additional data: a spin structure (when $\Char \mathbb{F}\neq 2$), and a rank $1$ local system over $\mathbb{F}$, i.e. a group homomorphism $\rho:\pi_1(L,x_0) \rightarrow \mathbb{F}^*$.  It is natural to ask how closely the algebra reflects the geometry.

\medskip

\begin{question} \label{question:1} Given a field $\mathbb{F}$, which augmentations of $\mathcal{A}(\Lambda)$ to $\mathbb{F}$ can be induced by Lagrangian fillings?
\end{question}

\medskip

Many augmentations cannot be induced by any embedded Lagrangian filling as can be seen by the following obstructions:

\begin{enumerate}
\item A Lagrangian filling must satisfy $\mathit{tb}(\Lambda) = 2g(L)-1$ where $\mathit{tb}(\Lambda)$ is the Thurston-Bennequin number of $\Lambda$.  See \cite{Chan}. 
\item The linearized cohomology of an augmentation associated to a filling  $L$ must be isomorphic (after shifting grading up by $1$) to the relative cohomology $H^*(L, \Lambda)$.  See \cite{Ekh, DR}. 
\item More recently, the work of Ekholm-Lekili \cite{EL} upgrades (a variant of) the above isomorphism to include an $A_\infty$-algebra structure, and Etg\"{u} \cite{Etgu} provides examples of augmentations where the $A_\infty$-algebra structure on linearized cohomology obstructs fillings in a somewhat subtle way.

\end{enumerate}
In this article, we use the local structure of the augmentation variety to obstruct Lagrangian fillings, and provide examples of augmentations of negative twist knots for which the obstruction applies while (1)-(3) do not.  In fact, we provide examples of non-fillable augmentations for any maximal Thurston-Bennequin number negative twist knot with odd crossing number $\geq 5$.

For a specific case, consider the family of Legendrian twist knots, $\Lambda_n$ where $n\geq 3$ is odd, pictured in Figure \ref{fig:Lambda}.  Eg., $\Lambda_3$ is one of the famous Chekanov-Eliashberg $m(5_2)$ knots; see \cite{ENV} for a complete classification of Legendrian twist knots.  Each $\Lambda_n$ has three augmentations to $\mathbb{F}_2$ given by
\[
\epsilon_1: \left(\begin{array}{ccc} a & \mapsto & 0 \\ b, c_1, \ldots, c_n & \mapsto & 1 \end{array} \right),    \quad  \epsilon_2:   \left(\begin{array}{ccc} b & \mapsto & 0 \\ a, c_1, \ldots, c_n & \mapsto & 1 \end{array} \right), \quad \epsilon_3: \left(\begin{array}{ccc} a, b & \mapsto & 0 \\  c_1, \ldots, c_n & \mapsto & 1 \end{array}\right).
\]
While both of the augmentations $\epsilon_1$ and $\epsilon_2$ can be induced by an embedded Lagrangian filling with $m(L)=0$, it is shown in Corollary \ref{cor:main} that $\epsilon_3$ cannot be.

\begin{figure}[!ht]
\labellist
\small
\pinlabel $e_0$ [l] at 222 226
\pinlabel $e_1$ [l] at 222 174
\pinlabel $e_2$ [l] at 222 138
\pinlabel $e_3$ [l] at 222 100
\pinlabel $e_4$ [l] at 222 54
\pinlabel $e_5$ [l] at 70 104
\pinlabel $e_6$ [l] at 70 144
\pinlabel $e_7$ [l] at 70 180
\pinlabel $c_1$ [l] at 196 154
\pinlabel $c_2$ [l] at 196 116
\pinlabel $c_3$ [l] at 196 76
\pinlabel $c_4$ [l] at 42 86
\pinlabel $c_5$ [l] at 42 124
\pinlabel $c_6$ [l] at 42 162
\pinlabel $c_7$ [l] at 50 198
\pinlabel $a$ [b] at 158 194
\pinlabel $b$ [b] at 196 202

\endlabellist
\includegraphics[scale=.8]{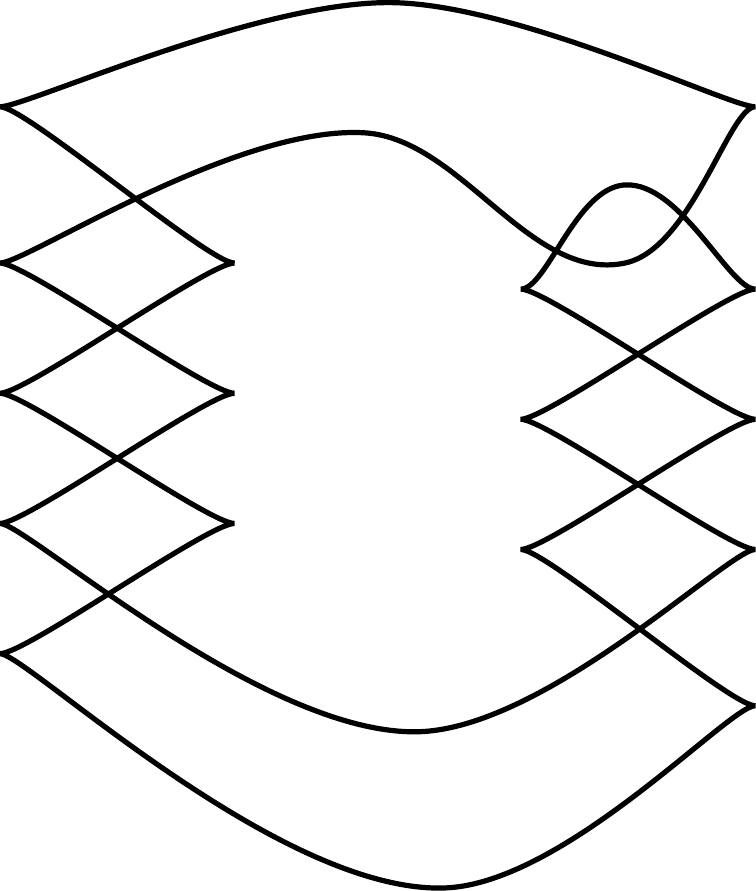}
\vspace{0.1in}
\caption{The Legendrian knot $\Lambda_n$, pictured in the front projection with $n =7$.  For $n = 2k+1 \geq 3$, crossings $a$, $b$ and $c_1, \ldots, c_k$ appear on the right, while crossings $c_{k+1}, \ldots, c_{n}$ appear on the left.  The right cusps are labeled in clockwise order starting at the upper right as $e_0, e_1, \ldots, e_{k+1}$ (appearing on the right) and $e_{k+2}, \ldots, e_n$ (appearing on the left).}
\label{fig:Lambda}
\end{figure}
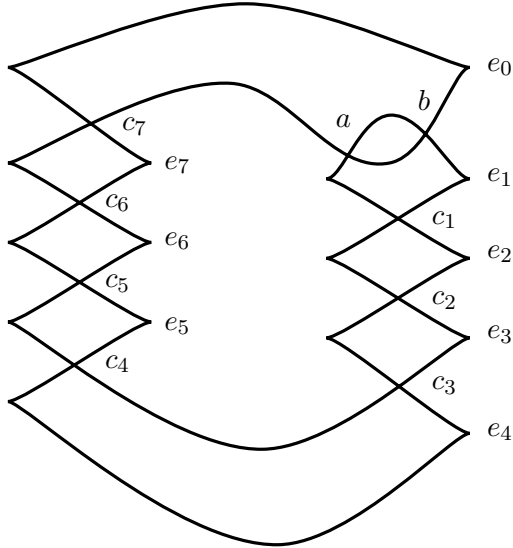

To more precisely describe the obstructions that we will apply, let $\mathit{Aug}_m(\Lambda, \mathbb{F})$ denote the $m$-graded augmentation variety of $\Lambda$.  It is an affine algebraic variety, i.e. an algebraic set in $\mathbb{F}^N$; see Section \ref{sec:Aug}.  In analogy with a result of Jin and Treumann from microlocal sheaf theory \cite{JT}, see also \cite{NadlerZ, Gui} and \cite[Section 2.3-2.4]{STWZ}, 
we have that: 
\begin{proposition} \label{prop:1} Assume $L \subset \mathit{Symp}(\R^3)$ is an exact orientable Lagrangian filling of $\Lambda \subset \R^3$ with Maslov number $m(L) =m$.  
If $\epsilon \in \mathit{Aug}_m(\Lambda, \mathbb{F})$ is induced by $L$ (via a choice of spin structure and rank $1$ local system, $\rho:\pi_1(L,x_0) \rightarrow \mathbb{F}^*$), then $\epsilon$ lies in the image of an injective, algebraic map 
\[
f_L^*: \Loc \cong (\mathbb{F}^*)^{\dim(H_1(L))} \hookrightarrow \mathit{Aug}_m(\Lambda, \mathbb{F}).
\]
Moreover, if $\mathit{Aug}_m(\Lambda, \mathbb{F})/{\sim}$ denotes the set of DGA homotopy classes of augmentations, then the composition 
\[
(\mathbb{F}^*)^{\dim(H_1(L))} \stackrel{f^*_L}{\hookrightarrow} \mathit{Aug}_m(\Lambda, \mathbb{F}) \rightarrow \mathit{Aug}_m(\Lambda, \mathbb{F})/{\sim}
\] 
is also injective.
\end{proposition}
A more general statement that applies in all dimensions appears as Proposition \ref{prop:main2}.  The proof 
involves observing an extension of the Seidel-Ekholm-Dimitroglou Rizell isomorphism, see \cite{Ekh,DR}, to Lagrangian fillings equipped with local coefficient systems.

To obstruct fillings inducing the specific augmentation, $\epsilon_3$, to $\mathbb{F}_2$ mentioned above,
we consider the corresponding point in the variety $\mathit{Aug}(\Lambda, \overline{\mathbb{F}}_2)$ over the algebraic closure of $\mathbb{F}_2$, and show that it cannot be in the injective image of such an algebraic torus.  See Corollary \ref{cor:main},
which also applies the obstruction to  
augmentations similar to $\e_3$ but defined over arbitrary fields.  Moreover, using the Etnyre-Ng-Vertesi classification from \cite{ENV} we generalize this example to produce non-fillable augmentations of arbitrary max-$\mathit{tb}$ Legendrian twist knots with the same topological type as the $\Lambda_n$. 

\begin{remark}
In \cite{NRSSZ} a quasi-equivalence of $A_\infty$-categories is established between an $A_\infty$-category, $\mbox{Aug}_+(\Lambda)$, whose objects are augmentations and a dg derived category of microlocal rank $1$ constructible sheaves with singular support specified by $\Lambda$.  The sheaf categories were introduced into the Legendrian knot theory in \cite{STZ}. In particular, on isomorphism classes of objects we get a bijection between $\mathit{Aug}(\Lambda, \mathbb{F})/{\sim}$ and the corresponding moduli space of sheaves, so that one could apply \cite{NRSSZ} together with \cite[Section 2.4]{STWZ} to deduce a statement about the local structure of the moduli space $\mathit{Aug}(\Lambda, \mathbb{F})/{\sim}$ near a Lagrangian filling, similar to Proposition \ref{prop:1}.  In this article, we establish Proposition \ref{prop:1} via a purely Floer-theoretic proof independent of \cite{NRSSZ}, though the approach, showing that there is a quasi-isomorphism between hom-spaces in $\mbox{Aug}_+(\Lambda)$ and in the category of local systems on $L$, is very much inspired by the outline on the sheaf side.  
The Floer-theoretic proof provides a uniform treatment of the case when $m(L)=m$ is even but not necessarily $0$ that is important for obstructing arbitrary orientable fillings, and also allows for the statement about the naive augmentation variety (not quotienting by DGA homotopy).
\end{remark}

Complementing the obstructions discussed above, there has also been much work involving constructions of Lagrangian fillings for different classes of Legendrian knots, sometimes realizing prescribed augmentation sets. We refer the interested reader to eg. \cite{EHK, BST, STWZ, HSab, LSab, CET, CG,  GSW, CN, ABL, Hughes}.  In another direction, if one allows Lagrangian fillings to have double points, then it is shown in \cite{PanRu2} that any augmentation to $\mathbb{F}_2$ can be induced in some appropriate manner by an immersed exact Lagrangian filling.

The contents of the remainder of the article are as follows:  Section \ref{sec:2} reviews the Legendrian contact DGA and augmentation varieties and provides a more general statement (Proposition \ref{prop:main2}) implying Proposition \ref{prop:1}.  Section \ref{sec:3} establishes a version of the Seidel-Ekholm-Dimitroglou Rizell isomorphism with local coefficients (Proposition \ref{prop:Seidel}) and then proves Proposition \ref{prop:main2}.  Section \ref{sec:lambda} computes the augmentation varieties of the twist knots, $\Lambda_n$, (Proposition \ref{prop:variety}), and then establishes that $\epsilon_3$ cannot be induced by any Lagrangian filling via Proposition \ref{prop:V}.  In Section \ref{sec:formality} we compute the $A_\infty$ structure on the linearized cohomology of $\e_3$ and show that the obstruction (3) does not apply. Section \ref{sec:general} considers the case of arbitrary negative Legendrian twist knots.

\subsection{Acknowledgements}  The first author thanks the hospitality of Ball State University during his visits in 2018 and 2019. The second author is grateful to AIM for hosting a SQuaREs meeting where the local structure of the augmentation variety near a Lagrangian filling, viewed from the perspective of micro-local sheaf theory, was discussed and to the participants Lenny Ng, Vivek Shende, Steven Sivek, David Treumann, and Eric Zaslow for generously sharing their knowledge.  Mohammed Abouzaid also formulated the possibility of using the local structure of the moduli space of augmentations to obstruct fillings during a problem session at a Workshop on Immersed Lagrangian Cobordisms at University of Ottawa sponsored by the Fields Institute.  Thanks to the organizers and participants for an interesting conference.  The first author is partially supported by an AMS-Simons travel grant. The second author is partially supported by grant 429536 from the Simons Foundation.

\section{Algebraic tori from Lagrangian fillings} \label{sec:2}

In this section we review the Legendrian contact DGA and augmentations varieties.  We describe in Proposition \ref{prop:main2} the subset of the augmentation variety induced by an exact Lagrangian filling.

\subsection{The Legendrian contact DGA} \label{sec:LCH}
We recall the Legendrian contact DGA with fully non-commutative $R[\pi_1(\Lambda,x_0)]$ and $R[H_1(\Lambda)]$-coefficients where {\it throughout $R= \mathbb{F}_2$ or $\mathbb{Z}$.}  

Let $\Lambda$ be a closed, connected $n$-dimensional Legendrian submanifold {\it with vanishing Maslov class}, i.e. with Maslov number $m(\Lambda) =0$, in\footnote{More generally, all results in Sections \ref{sec:2}-\ref{sec:3} hold for $\Lambda$ in $P\times \R$, the contactization of an exact symplectic manifold, as in the setting of \cite{EES07, CDGG}.}   
a $1$-jet space $J^1M$.  We use $(x,y,z) \in J^1M = T^*M \times \R$ for local coordinates arising from a coordinate $x \in M$.  With respect to the standard contact form, $dz-\sum_iy_i\,dx_i$, the Reeb vector field is $\frac{\partial}{\partial z}$.   The Reeb chords of $\Lambda$ are in bijection with double points of the Lagrangian  projection $\pi_{xy}:\Lambda \rightarrow T^*M$, and we write $\Reeb(\Lambda)$ for the set of Reeb chords of $\Lambda$.  
Choose a base point $x_0 \in \Lambda$, and for each $q \in \Reeb(\Lambda)$ choose {\bf base point paths}  
\[
\gamma_q^+, \gamma_q^-:[0,1] \rightarrow \Lambda \quad \mbox{with} \quad \gamma_q^\pm(0) = q^\pm, \, \gamma_q^\pm(1)=x_0.
\]
where $q^+, q^-$ are the upper and lower Reeb chord endpoints of $q$, i.e. $z(q^+)>z(q^-)$.  
Each $q \in \Reeb(\Lambda)$ is assigned an integer grading that is $1$ less than a certain Conley-Zhender index. See \cite{EES05a, EES07} for more detail.  

We denote by $\mathcal{A}(\Lambda)$ the {\bf Legendrian contact DGA} (differential graded algebra) of $\Lambda$ with fully non-commutative coefficients in the group algebra $R[\pi_1(\Lambda,x_0)]$.  It is the unital graded associative (but non-commutative) $R$-algebra generated by $\pi_1(\Lambda, x_0) \cup \Reeb(\Lambda)$ where the only relations are from $\pi_1(\Lambda, x_0)$ and that the identity element of $\pi_1(\Lambda,x_0)$ is the unit of $\mathcal{A}(\Lambda)$. The group algebra $R[\pi_1(\Lambda,x_0)]$ sits in graded degree $0$ as a sub-algebra of $\mathcal{A}(\Lambda)$. A basis for $\mathcal{A}(\Lambda)$ as an $R$-module consists of words of the form
\[
\alpha_0 q_1 \alpha_1 \cdots  q_n \alpha_n, \quad  \quad \alpha_0, \ldots, \alpha_n \in \pi_1(\Lambda), \quad q_1, \ldots, q_n \in \Reeb(\Lambda), \quad n \geq 0.
\]
In the case $R = \Z$, $\Lambda$ must be equipped with a choice of spin structure, $\xi$. When necessary we write $\mathcal{A}(\Lambda;\mathbb{F}_2[\pi_1])$ or $\mathcal{A}(\Lambda; \mathbb{Z}[\pi_1])$ to distinguish between the case of $R = \mathbb{F}_2$ or $\mathbb{Z}$.  

As in \cite{EES05a, EES07, EES05b, DR}, the differential can be defined by counting transversally cut out, $0$-dimensional (i.e., ``rigid'') moduli spaces of holomorphic disks in the symplectization $\mathit{Symp}(J^1M)= \R \times J^1M$ with boundary on $\R \times \Lambda$ and having boundary punctures that asymptotically approach Reeb chords at the positive and negative ends of $\mathit{Symp}(J^1M)$.  For more details and for a discussion of the class of almost complex structures allowed, consult the original sources, loc. cit.  In the Legendrian contact DGA only disks with a single positive puncture are used.  Consider such a holomorphic disk, 
\[
u:D^2 \setminus \{p_0,p_1, \ldots, p_n\} \rightarrow \R\times J^1M 
\]
where $p_0, \ldots, p_n$ appear in counter-clockwise order around $\partial D^2$, $p_0$ is a positive puncture at $a \in \Reeb(\Lambda)$, and $p_1, \ldots, p_n$ are negative punctures at $b_1, \ldots, b_n \in \Reeb(\Lambda)$. At the positive puncture, as $z\in \partial D^2$ approaches $p_0$ from the clockwise (resp. counter-clockwise) direction $u(z)$ projected to $\Lambda$  converges to $a^-$ (resp. $a^+$); whereas at negative punctures, $p_i$, the clockwise (resp. counter-clockwise) limit will be $b_i^+$ (resp. $b_i^-$). We associate elements
\[
\beta_0, \ldots, \beta_n \in \pi_1(\Lambda, x_0)
\]
to $u$ by concatenating with base point paths
\begin{align*}
\beta_0 &= (\gamma_a^+)^{-1} * u_{[p_0,p_1]}* \gamma_{b_1}^+, \\
\beta_i &=(\gamma_{b_i}^-)^{-1} * u_{[p_i,p_{i+1}]}* \gamma_{b_{i+1}}^+, \quad 1\leq i \leq n-1, \\
\beta_n &=(\gamma_{b_n}^-)^{-1} * u_{[p_n,p_0]}* \gamma_a^-,
\end{align*}
where we write here (and elsewhere) $u_{[p_i,p_{i+1}]}$ for the path on $\Lambda$ obtained from restricting $u$ to the counter-clockwise boundary arc from $p_i$ to $p_{i+1}$ and adding in the $1$-sided limits at Reeb chord endpoints. Writing 
\[
\mathbf{b} = \beta_0 b_1 \beta_1 \cdots b_n \beta_n,
\]
we let $\mathcal{M}(a; \mathbf{b})$ denote the moduli space of disks with boundary conditions specified as above by $a$ and the word $\mathbf{b}$, obtained from quotienting by translation in the $\R$ direction in $\mathit{Symp}(J^1M)$ and by any automorphisms of the domain.  See Figure \ref{fig:Disks}.  When $\Lambda$ is equipped with a spin structure, $\xi$, these moduli spaces are coherently oriented as in \cite{EES05b, EES07}.  The differential is defined on generators by
\[
\partial a = \sum_{\dim \mathcal{M}(a;\mathbf{b})=0} \#\mathcal{M}(a;\mathbf{b}) \cdot \mathbf{b}
\] 
where the sum is over all such rigid moduli spaces.  The coefficient $\#\mathcal{M}(a;\mathbf{b})$ is either a signed count when $R =\Z$ or a mod $2$ count when $R=\F_2$. With $\partial:\mathcal{A}(\Lambda) \rightarrow \mathcal{A}(\Lambda)$ extended to the full algebra as a degree $-1$ derivation, $(\mathcal{A}(\Lambda), \partial)$ is an (associative) DGA, whose stable tame isomorphism type is a Legendrian isotopy invariant of $(\Lambda, \xi)$.  

One can also consider the Legendrian contact DGA with fully non-commutative {\bf homology coefficients}, notated here as $\mathcal{A}(\Lambda;R[H_1(\Lambda)])$, that is obtained by specializing $R[\pi_1(\Lambda,x_0)] \rightarrow R[H_1(\Lambda)]$ using the canonical homomorphism $\pi_1(\Lambda, x_0) \rightarrow H_1(\Lambda) := H_1(\Lambda; \Z)$.  

\begin{remark} Homology coefficients have been used since the construction of Legendrian contact homology in \cite{EES05a}, though they do not appear in Chekanov's original combinatorial definition for knots in $\R^3$ from \cite{Ch} that is upgraded to use $\Z[H_1(\Lambda)]$ coefficients in \cite{ENS}.  
The $R[\pi_1(\Lambda, x_0)]$ coefficients are used in \cite{CDGG1, CDGG}.  See \cite{EN} for a recent survey on Legendrian contact homology.
\end{remark}

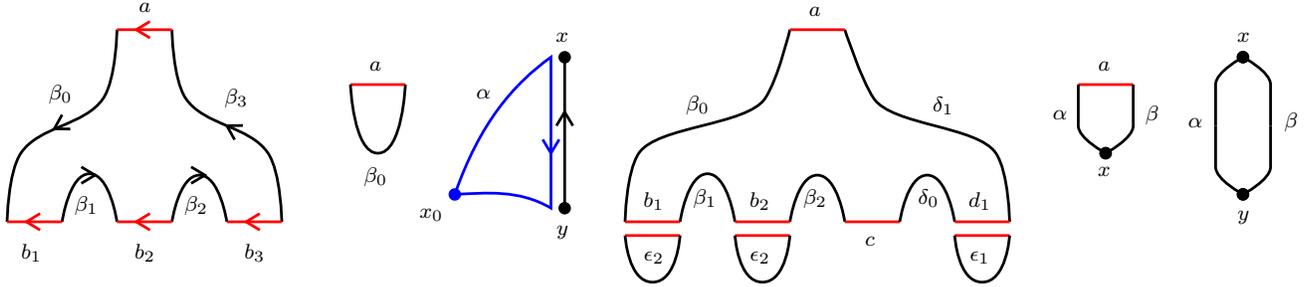
\begin{figure}[!ht]
\labellist
\tiny
\pinlabel $a$ [b] at 76 144
\pinlabel $\beta_0$ [b] at 30 94
\pinlabel $b_1$ [t] at 14 26
\pinlabel $\beta_1$ [t] at 44 52
\pinlabel $b_2$ [t] at 76 26
\pinlabel $\beta_2$ [t] at 104 52
\pinlabel $b_3$ [t] at 136 26
\pinlabel $\beta_3$ [b] at 126 92

\pinlabel $a$ [b] at 202 112
\pinlabel $\beta_0$ [t] at 202 68

\pinlabel $x_0$ [tr] at 242 44
\pinlabel $y$ [t] at 304 36
\pinlabel $x$ [b] at 304 128
\pinlabel $\alpha$ [br] at 268 98

\pinlabel $a$ [b] at 442 142
\pinlabel $\beta_0$ [b] at 378 88
\pinlabel $b_1$ [b] at 354 36
\pinlabel $\beta_1$ [t] at 382 56
\pinlabel $b_2$ [b] at 412 36
\pinlabel $\beta_2$ [t] at 442 56 
\pinlabel $c$ [t] at 472 30
\pinlabel $\delta_0$ [t] at 504 56
\pinlabel $d_1$ [b] at 532 36
\pinlabel $\delta_1$ [b] at 512 88
\pinlabel $\e_2$  at 354 14
\pinlabel $\e_2$  at 412 14
\pinlabel $\e_1$  at 532 14

\pinlabel $a$ [b] at 600 112
\pinlabel $\alpha$ [r] at 584 92
\pinlabel $x$ [t] at 600 68
\pinlabel $\beta$ [l] at 618 92

\pinlabel $x$ [b] at 676 128
\pinlabel $\alpha$ [r] at 658 88
\pinlabel $y$ [t] at 676 44
\pinlabel $\beta$ [l] at 694 88

\endlabellist
\includegraphics[scale=.7]{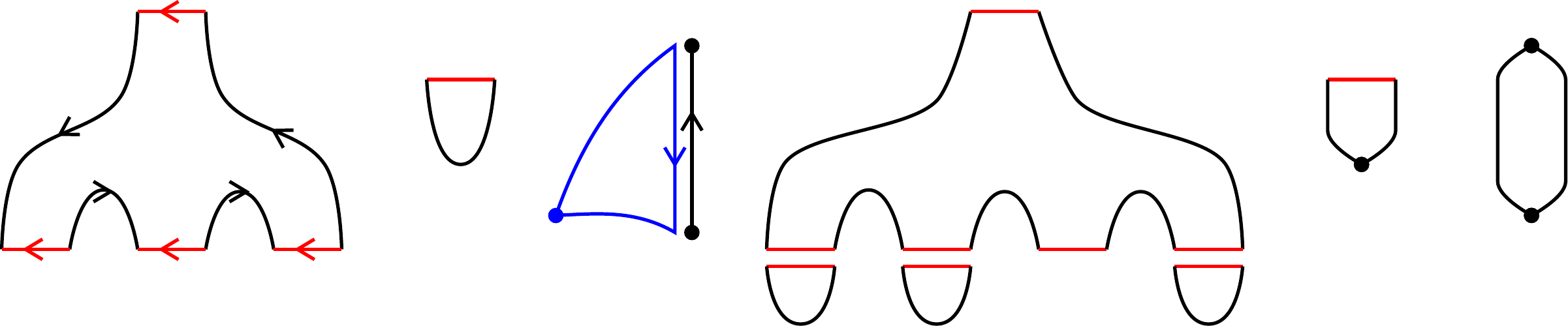}
\vspace{0.1in}
\caption{Schematic depictions of elements of the rigid moduli spaces used to define differentials and homomorphisms.  From left to right the moduli spaces are $\mathcal{M}(a; \mathbf{b})$ (Legendrian contact DGA),  $\mathcal{M}(a; \beta)$ (augmentation induced by a Lagrangian filling), $\mathcal{M}_f(x,y; \alpha)$ with the loop $\alpha$ in blue (Morse complex over $R[\pi_1(M,x_0)]$), $\mathcal{M}(a; \mathbf{b}, c, \mathbf{d})$ (linearized cochain complex and $d_\infty$ in $CF^*$),  $\mathcal{M}(a; \alpha, x, \beta)$ (the map $\phi$ in $CF^*$), and $\mathcal{M}(x; \alpha, y, \beta)$ ($d_0$ in $CF^*$).}
\label{fig:Disks}
\end{figure}

\subsection{Augmentation varieties}  \label{sec:Aug}
Let $S$ be a commutative, ring with identity element.  We view $S$ as being $\Z$-graded and concentrated in degree $0$.  An {\bf augmentation} of $\mathcal{A}(\Lambda)$ is a unital, graded ring homomorphism
\[
\epsilon:\mathcal{A}(\Lambda) \rightarrow S   \quad \mbox{satisfying} \quad \epsilon \circ \partial = 0.
\]
Denote by $\mathit{Aug}(\Lambda, S)$ the set of all augmentations to $S$.

When $\mathbb{F}$ is a field, we refer to $\mathit{Aug}(\Lambda, \mathbb{F})$ as the {\bf augmentation variety} of $\Lambda$ over $\mathbb{F}$.  It is an affine variety:  Choosing an ordering of the degree $0$ Reeb chords $\Reeb_0(\Lambda) = \{q_1, \ldots, q_n\}$ and a group presentation $\pi_1(\Lambda, x_0) = \langle g_1, \ldots, g_m \,|\, r_1, \ldots, r_s\rangle$ 
produces an explicit embedding 
\[
\begin{array}{rcl}
\mathit{Aug}(\Lambda, \mathbb{F}) & \stackrel{\cong}{\rightarrow} & \overline{\mathit{Aug}(\Lambda, \mathbb{F})} \subset \mathbb{F}^n\times(\mathbb{F}^*)^m, \\
\epsilon  & \mapsto & (\epsilon(q_1),\ldots, \epsilon(q_n),\epsilon(g_1),\ldots, \epsilon(g_m)).
\end{array}
\]
Writing $\Reeb_1(\Lambda) = \{p_1, \ldots, p_k\}$ for the degree $1$ Reeb chords of $\Lambda$, the image, $\overline{\mathit{Aug}(\Lambda, \mathbb{F})}$, is the zero set of the polynomial equations 
\[
\partial p_1=0, \ldots, \partial p_k =0,  r_1=1, \ldots, r_s=1 
\]
where we obtain the Laurent polynomials $\partial p_i, r_i \in \mathbb{F}[Q_1, \ldots, Q_n, G^{\pm1}_1,\ldots, G^{\pm1}_m]$ from the differentials $\partial p_i$ and relations $r_i$ by abelianizing; choosing a factorization of each $\pi_1$ element appearing in $\partial p_i$ in terms of the generators $\{g_1, \ldots, g_m\}$; and then replacing lower case letters with capitals.

\begin{remark}
\begin{enumerate}
\item We have presented the definition using $R[\pi_1(\Lambda, x_0)]$ coefficients in $\mathcal{A}(\Lambda)$, but the same variety arises using homology coefficients. For considering augmentations valued in non-commutative DGAs, eg. higher rank representations, the choice of working with $\pi_1$ or $H_1$ coefficients can lead to distinct augmentation varieties.
\item When $\Char \mathbb{F}=2$, it suffices to use $\mathcal{A}(\Lambda;\mathbb{F}_2[\pi_1(\Lambda,x_0)])$ when defining $\mathit{Aug}(\Lambda, \mathbb{F})$.
\item When $\Char \mathbb{F} \neq 2$, $R=\Z$ should be used.  However, the dependence on the choice of spin structure $\xi$ is negligible.  With homology coefficients, the DGAs arising from a different choice of spin structure are related by a  DGA isomorphism having the form $t_i \mapsto \pm t_i$ on a generating set for $\Z[H_1(\Lambda)]$ and restricting to the identity on Reeb chord generators; see \cite[Theorems 4.29 and 4.30]{EES05b}.  In particular, the augmentation varieties associated to distinct spin structures are isomorphic; though their projections to $(\mathbb{F}^*)^m$ need not coincide.
\end{enumerate}
\end{remark}

As augmentations can be viewed as DGA homomorphisms, $(\mathcal{A}(\Lambda),\partial) \rightarrow (\mathbb{F}, 0)$, a natural equivalence relation on $\mathit{Aug}(\Lambda, \mathbb{F})$ is provided by DGA homotopy.  Here, two augmentations $\epsilon_1,\epsilon_2 \in \mathcal{A}(\Lambda, \mathbb{F})$ are {\bf DGA homotopic} if there exists a degree $1$ $(\epsilon_1, \epsilon_2)$-derivation
\[
K: \mathcal{A}(\Lambda) \rightarrow \mathbb{F}, \quad K(xy) = K(x) \e_2(y) + (-1)^{|x|} \e_1(x) K(y)
\]  
satisfying
\[
\epsilon_1-\epsilon_2 = K \circ \partial.
\]
Although it is not obvious, DGA homotopy defines an equivalence relation on $\mathit{Aug}(\Lambda, \F)$, see eg. \cite{FHT}.  We denote the set of DGA homotopy classes of augmentations as $\mathit{Aug}(\Lambda,\F)/{\sim}$.  

\begin{remark}  For $m \in \Z_{\geq 0}$ even, we will also consider {\bf $m$-graded augmentations} which are ring homomorphisms $\epsilon: \mathcal{A}(\Lambda) \rightarrow S$ that only preserve grading mod $m$, i.e. an $m$-graded augmentation satisfies $\epsilon(1) =1$, $\epsilon \circ \partial =0$, and $q_i \in \mathcal{R}(\Lambda)$ must have $\epsilon(q_i) = 0$ unless $|q_i| = 0$ (mod $m$).  In the same manner as above, we can then consider the {\bf $m$-graded augmentation variety}, $\mathit{Aug}_m(\Lambda, \mathbb{F})$, and its quotient by DGA homotopy, $\mathit{Aug}_m(\Lambda, \mathbb{F})/{\sim}$, where DGA homotopy operators, $K$, are now only required to have degree $1$ mod $m$.  Note that the case $m=0$ corresponds to ordinary augmentations.
\end{remark}

\subsection{Augmentations induced by exact Lagrangian fillings}
  Let $L \subset \mathit{Symp}(J^1M)$  be an exact Lagrangian filling 
	of $\Lambda \subset J^1M$  such that $m(L)= 0$, i.e. $L$ agrees with $\R \times \Lambda$ near the positive end of the symplectization and vanishes near the negative end.    
	As in \cite{Ekh2} and using the upgrade to $\pi_1$ or $H_1$-coefficients as in \cite{Pan2}, $L$ induces an $m$-graded augmentation 
\begin{equation} \label{eg:fL}
f_L: \mathcal{A}(\Lambda, R[\pi_1(\Lambda, x_0)]) \rightarrow R[\pi_1(L, x_0)], \quad f_L \circ \partial =0 
\end{equation}
where when $R= \Z$, $L$ should be equipped with a spin structure extending the given spin structure on  $\Lambda$; see \cite{K}.
The homomorphism is defined on the coefficient rings to be the map $\Z[\pi_1(\Lambda,x_0)] \rightarrow \Z[\pi_1(L, x_0)]$ induced by the inclusion $\{t_0\} \times \Lambda \hookrightarrow L$ at some $t_0 \gg 0$ and on generators by
\[
f_L(a) = \sum_{\dim \mathcal{M}(a; \beta) =0} \#\mathcal{M}(a; \beta) \cdot \beta.
\]
The sum is over $0$-dimensional moduli spaces $\mathcal{M}(a;\beta)$ of holomorphic disks, $u$, mapped to $\mathit{Symp}(J^1M)$  with boundary on $L$, having a single boundary puncture $p_0$ that is asymptotic to the Reeb chord $a$ at the positive end of $\mathit{Symp}(J^1M)$, and such that
\[
\beta = [ (\gamma_a^+)^{-1} *  u|_{[p_0,p_0]} * \gamma_a^-] \in \pi_1(L, x_0).
\]

Augmentations to a field, $\mathbb{F}$, then arise from a choice of group homomorphism, $\rho:\pi_1(L, x_0) \rightarrow \mathbb{F}^*$, or  equivalently a ring homomorphism
\[
\rho:R[\pi_1(L, x_0)] \rightarrow \mathbb{F},  
\]
where $R = \mathbb{Z}$ unless $\Char \mathbb{F} =2$. We refer to the composition
\[
\epsilon_{(L, \rho)} := f^*_L \rho = \rho \circ f_L
\]
as the {\bf augmentation of $\Lambda$ induced by $(L,\rho)$}.

\begin{remark} 
Viewing $\mathbb{F}^* = GL(1,\mathbb{F})$, such a group homomorphism $\rho$ is also equivalent to a local system  of rank $1$ $\mathbb{F}$-vector spaces on $L$.
\end{remark}

\begin{remark}
If the Maslov number $m(L) =m \geq 0$ is even, but not necessarily $0$, then $f_L$ and also the $\epsilon_{(L,\rho)}$ only need to preserve grading mod $m$.  I.e., they are $m$-graded.  If $M$ is orientable, then such a filling $L \subset \mathit{Symp}(J^1M)$ is orientable if and only if $m$ is even;  see eg. \cite[Remark 2.9]{PanRu2}.
\end{remark}
  
In our context, it is natural to think of $R[\pi_1(L,x_0)]$ as the Legendrian contact DGA of $L$ (after all, exact Lagrangian submanifolds lift to Legendrians submanifolds in the contactization), so we will write $\mathit{Aug}(L, \mathbb{F})$ for the set of all ring homomorphisms $\rho: R[\pi_1(L,x_0)] \rightarrow \mathbb{F}$.  It is also an affine variety:  fixing an isomorphism
\[
H_1(L) \cong \Z^r \oplus \Z/n_1 \oplus \cdots \oplus \Z/n_s
\]
we get
\[
\mathit{Aug}(L, \mathbb{F}) \cong (\mathbb{F}^*)^{r}  \times Z_{n_1}(\mathbb{F}) \times \cdots  \times Z_{n_s}(\mathbb{F})
\]
where $Z_{n}(\mathbb{F}) := \{x \in \mathbb{F}\,|\, x^n =1\}$ is the group of $n$-th roots of unity in $\mathbb{F}$.  

\begin{proposition} \label{prop:main2} Let $L\subset \mathit{Symp}(J^1M)$ be a Lagrangian filling of $\Lambda$ with Maslov number $m(L) =0$.  If $\Char \F \neq 2$, then let $L$ be equipped with a choice of spin structure.  Then, the map
\[
f_L^*: \mathit{Aug}(L, \mathbb{F}) \rightarrow \mathit{Aug}(\Lambda, \mathbb{F})
\]
is an injective, algebraic map.

Moreover, assuming $\dim L =2$,  allowing for any even $m(L) = m \in \Z_{\geq 0}$, 
the statement applies to $f^*_L:\mathit{Aug}(L, \mathbb{F}) \rightarrow \mathit{Aug}_m(\Lambda, \mathbb{F})$, and the composition 
\[
\mathit{Aug}(L, \mathbb{F}) \stackrel{f^*_L}{\hookrightarrow} \mathit{Aug}_m(\Lambda, \mathbb{F}) \rightarrow \mathit{Aug}_m(\Lambda, \mathbb{F})/{\sim}
\] 
with the projection to DGA homotopy classes is also injective.
\end{proposition}

The proposition will be proved at the end of the next section.

Note that when $\Lambda \subset \R^3 = J^1\R$ is one dimensional and $L$ is orientable, we have that $m(L)$ must be even and  $H_1(L)$ is free abelian.  Thus, the augmentation map has the form
\[
f_L^*: (\mathbb{F}^*)^{\dim H_1(L)} \hookrightarrow \mathit{Aug}_m(\Lambda, \mathbb{F})
\] 	
matching the statement of Proposition \ref{prop:1} from the introduction.

\section{Seidel-Ekholm-Dimitroglou Rizell isomorphism with local coefficients}	 \label{sec:3}
In this section, we establish a version of the Seidel-Ekholm-Dimitroglou Rizell isomorphism with local coefficients, see Proposition \ref{prop:Seidel}.  At the conclusion of the section, the isomorphism is then applied to prove Proposition \ref{prop:main2}.  As preparation we recall a version of Morse cohomology with local coefficients in \ref{sec:Morse}.  The version of linearized cohomology we work with is from the positive augmentation category; see Sections \ref{sec:lincoh} and \ref{sec:positive} for quick definitions.  Although we focus on rank $1$ local systems and augmentations, a version of Proposition \ref{prop:Seidel} holds for higher rank local systems and representations of the Legendrian contact DGA; it is explained in a series of remarks.   

\subsection{Morse complex with local coefficients}  \label{sec:Morse} Let $M$ be a manifold and  $f:M \rightarrow \R$ be a proper, bounded below Morse function with finitely many critical points, $\mathit{Crit}(f)$.  
To define a version of the Morse cohomology complex suitable for use with local coefficients, fix a basepoint $x_0 \in M$ and for each $x \in \mathit{Crit}(f)$ choose a basepoint path $\gamma_x:[0,1] \rightarrow M$ from $x$ to $x_0$.  With $R = \mathbb{F}_2$ or $\Z$, let 
\[
C^*(M,f; R[\pi_1]) =  \mathit{Span}_{R[\pi_1(M, x_0)]^\mathit{op}} \mathit{Crit}(f)
\]
be the free {\it right} $R[\pi_1(M,x_0)]$-module generated by $\mathit{Crit}(f)$ with $\Z$-grading from the Morse index of critical points.  To define the differential, fix a choice of metric so that the gradient vector field $\nabla f$ satisfies the Morse-Smale condition.  For critical points $x,y \in \mathit{Crit}(f)$ with $|x| = |y|+1$ and $\alpha \in \pi_1(M,x_0)$ let $\mathcal{M}_f(x,y; \alpha)$ denote the moduli space of {\it ascending} gradient trajectories, $\eta: \R \rightarrow M$, $\eta' = \nabla f \circ \eta$, from $y$ to $x$  such that
\[
\alpha = [\gamma_x^{-1} * \eta^{-1} * \gamma_y] \in \pi_1(M, x_0) 
\]
is the homotopy class arising from the corresponding {\it descending} trajectory from $x$ to $y$.  See Figure \ref{fig:Disks}.   
Then, the differential $d:C^*(M,f; R[\pi_1]) \rightarrow  C^{*+1}(M,f; R[\pi_1])$ is the right $R[\pi_1(M,x_0)]$-module homomorphism satisfying
\[
d y = \sum_{\dim \mathcal{M}_f(x,y; \alpha) =0} \#\mathcal{M}_f(x,y; \alpha)\cdot x \cdot \alpha.
\]
When $R = \mathbb{Z}$, to make the signed count of $\mathcal{M}_f(x,y;\alpha)$, the moduli spaces are coherently oriented in a standard way by a choice of orientations for $M$ and the descending manifolds of the critical points of $f$;  see eg. \cite[Section 5.2]{EES07}.

Next, we define the  Morse cohomology complex, $C^*(M,f; \rho)$, with respect to a local system, 
\[
\rho:\pi_1(M, x_0) \rightarrow \mathbb{F}^*.
\]
The representation $\rho$ makes $\mathbb{F}$ into a left $R[\pi_1(M,x_0)]$-module and we define
\[
C^*(M,f; \rho) =  C^*(M,f; R[\pi_1]) \otimes_{R[\pi_1(M,x_0)]} \mathbb{F}
\]
with the differential $d \otimes 1$.  In other words, $C^*(M,f; \rho)= \mathit{Span}_\mathbb{F}\mathit{Crit}(f)$ with 
\[
d y = \sum \#\mathcal{M}_f(x,y; \alpha)\cdot \rho(\alpha) \cdot x.
\]

The Morse cohomology, $H^*(M,f; \rho)$, is isomorphic to the standard singular cohomology (or sheaf cohomology) with respect to the local coefficient system specified by $\rho$; see \cite{BHS}. The following is then a standard fact.

\begin{proposition} \label{prop:standard}
Suppose that $M$ is connected, and let $\rho:\pi_1(M, x_0) \rightarrow \mathbb{F}^*$.  Then, 
\[
H^0(M,f;\rho) \cong \left\{\begin{array}{cr} \mathbb{F}, &  \mbox{$\rho$ is trivial}, \\  0, & \mbox{else}. \end{array} \right. 
\]
\end{proposition}
\begin{proof}
We can assume $f$ has a single local minimum $m$.  Let $y_1, \ldots, y_n$ denote the Morse index $1$ critical points of $f$.  Assigning a choice of orientations to the descending manifolds of the $y_1, \ldots, y_n$ produces  a generating set, $Y_1, \ldots, Y_n$, for $\pi_1(M,x_0)$ where we take $x_0 = m$.  Each of the $y_i$ is connected to $m$ by two gradient trajectories; using an appropriate one of them for the base point path, $\lambda_{y_i}$, leads to
\[
d\, m = \sum_i (1-\rho(Y_i)) y_i.
\]
Thus, $H^0(M,f;\rho) = \ker (d:C^0 \rightarrow C^1)$ is $0$ unless $1=\rho(Y_i)$ for all $i$. 
\end{proof}

\begin{remark} [Higher rank]  Given a higher rank local system $\rho: \pi_1(M,x_0) \rightarrow GL(V)$, we can view $V$ as a left $R[\pi_1(M,x_0)]$-module and then define the complex of $\mathbb{F}$ vector spaces 
\[
C^*(M,f; \rho) = C^*(M, f; R[\pi_1]) \otimes_{R[\pi_1(M,x_0)]} V.
\]
The differential looks like
\begin{equation} \label{eq:MorseHigher}
d( y \otimes v ) = \sum \#\mathcal{M}_f(x,y; \alpha) x \otimes (\rho(\alpha) v).
\end{equation}
\end{remark}

\subsection{Linearized cohomology} \label{sec:lincoh}
Let $\Lambda =\Lambda_1\cup \Lambda_2$ be a two component  Legendrian in $J^1M$ and let $\epsilon_i \in \mathit{Aug}(\Lambda_i, \mathbb{F})$.  With this data we form a {\bf linearized cochain complex} that we notate as, $C^*_{1,2}(\epsilon_1,\epsilon_2)$.  Said more precisely, it is the $(1,2)$-component in the link splitting of the linearized cohomology complex for $\Lambda$ with respect to the pure augmentation $\epsilon = (\e_1,\e_2)$; see eg. \cite{Mish, NRSSZ, CDGG}.  We review the definition.

Write $\Reeb_{1,2}(\Lambda)$ for the Reeb chords of $\Lambda$ with initial point on $\Lambda_1$ and ending point on $\Lambda_2$, i.e.
\[
\Reeb_{1,2}(\Lambda) = \{ q \in \Reeb(\Lambda) \,|\, q^+ \in \Lambda_2, \, q^- \in \Lambda_1 \},
\]
 and put
\[
C^*_{1,2}(\epsilon_1,\epsilon_2) = \mathit{Span}_{\mathbb{F}} \Reeb_{1,2}(\Lambda).
\]
We consider moduli spaces of disks as in the Legendrian contact DGA of the form 
\[
\mathcal{M}(a; \mathbf{b}, c, \mathbf{d})
\] 
where
\[
\begin{array}{cll} a,c \in \Reeb_{1,2}(\Lambda),   & \mathbf{b} = \beta_0b_1\beta_1\cdots b_k\beta_k, \quad &\mathbf{d} = \delta_0d_1\delta_1\cdots d_l\delta_l \\
& & \\
\mbox{with} & \beta_i \in \pi_1(\Lambda_2), \, b_i \in \Reeb(\Lambda_2), \quad 
  &\delta_i \in \pi_1(\Lambda_1), \, d_i \in \Reeb(\Lambda_1).
\end{array}
\]
For such disks, the puncture at $a$ is positive while the punctures at all other Reeb chords are negative;  the $\beta_i$ and $\delta_i$ are concatenations of counter-clockwise boundary segments with basepoint paths in the expected manner, cf. Section \ref{sec:LCH}.
Note that the counterclockwise portion of the boundary from the $a$ puncture to the $c$ puncture is mapped to $\Lambda_2$ (with possibile  punctures at Reeb chords of $\Lambda_2$), and the other half of the boundary is mapped to $\Lambda_1$ (again with possible punctures at Reeb chords).  The differential on $C^*_{1,2}(\epsilon_1,\epsilon_2)$ is then defined 
by summing over all $0$-dimensional moduli spaces and applying $\epsilon_i$ to the terms in $\mathcal{A}(\Lambda_i)$,
\begin{equation} \label{eq:lin}
d \,c = \sum \# \mathcal{M}(a; \mathbf{b}, c, \mathbf{d}) \cdot 
\epsilon_1(\mathbf{d})\, a\, \epsilon_2( \mathbf{b}).
\end{equation}
In contrast to the Legendrian contact DGA, $C^*(\epsilon_1,\epsilon_2)$ is a cohomologically graded complex, $d:C^*(\epsilon_1,\epsilon_2) \rightarrow C^{*+1}(\e_1,\e_2)$, and the differential maps negative punctures to positive punctures.\footnote{A mnemonic: For all of the maps defined by disk counts the output is obtained from reading the boundary conditions around the disk {\it counter-clockwise} from the puncture corresponding to the input element.  In notation for moduli spaces, we always start with the positive puncture (which may or may not be the input element of the map).}  See Figure \ref{fig:Disks} for a schematic depiction.

\begin{remark} [Higher rank]  The linearized cohomology generalizes to higher dimensional representations as follows, cf. \cite{CNS,CDGG1}.  Suppose now that $V_1$ and $V_2$ are $\mathbb{F}$-vector spaces with representations
\[
\phi_i: \mathcal{A}(\Lambda_i) \rightarrow \mathit{End}(V_i), \quad \phi_i(1) = \mbox{id}, \quad \phi_i \circ \partial=0.
\]
A linearized cochain complex can then be defined as 
\[
C^*_{1,2}(\phi_1,\phi_2) = V_1^\vee \otimes_\mathbb{F} \mathit{Span}_{\mathbb{F}} \Reeb_{1,2}(\Lambda) \otimes_\mathbb{F} V_2
\]
with differential satisfying for $\xi_1 \in V_1^\vee, c \in \Reeb_{1,2}(\Lambda)$, and $v_2 \in V_2$, 
\begin{equation} \label{eq:LinHigher}
d(\xi_1 \otimes c \otimes v_2) = \sum \# \mathcal{M}(a; \mathbf{b}, c, \mathbf{d}) \cdot 
(\phi_1(\mathbf{d})^\vee \xi_1)\otimes a \otimes(\phi_2( \mathbf{b})v_2).
\end{equation}

\end{remark}

\subsection{The positive augmentation category} \label{sec:positive}
Now we return to the case where $\Lambda \subset J^1M$ has a single component.  In \cite{NRSSZ} a unital $A_\infty$-category, $\mbox{Aug}_+(\Lambda; \mathbb{F})$, having $\mathit{Aug}(\Lambda, \mathbb{F})$ as its underlying set of objects  is constructed for $1$-dimensional Legendrian knots.  In particular, for each pair of augmentations $\epsilon_1, \epsilon_2 \in \mathit{Aug}(\Lambda, \mathbb{F})$ there is a {\bf hom-complex} in $\mbox{Aug}_+(\Lambda; \mathbb{F})$ denoted by
$\mathit{Hom}_+^*(\epsilon_1,\epsilon_2)$.  These complexes work (though not the higher $A_\infty$-operations) without requiring additional discussion for Legendrians of any dimension, and they arise as a special case of the above construction:  Form a two component link with $\Lambda_1 = \Lambda$ and $\Lambda_2 = \Lambda_{+,F}$ where $\Lambda_{+,F}$ is obtained from shifting $\Lambda$ a small amount in the positive Reeb direction (add a small $\delta>0$ to the $z$-coordinate) and then perturbing using a Morse function $F:\Lambda\rightarrow \R$ as in \cite[Section 4]{NRSSZ} or \cite[Section 6]{DR}.   With this perturbation, the $\Reeb_{1,2}$ Reeb chords are in bijection with the original Reeb chords of $\Lambda$ together with $\mathit{Crit}(F)$. With the perturbation small enough, the Legendrian contact DGA of $\Lambda_{+,F}$ is canonically identified with that of $\Lambda$.  Thus, we can consider $(\epsilon_1,\epsilon_2)$ as augmentations for the pair $(\Lambda_1, \Lambda_2)$ and set
\[
\mathit{Hom}_+^*(\epsilon_1,\epsilon_2) = C^{*-1}_{1,2}(\epsilon_1, \epsilon_2).
\]

\subsection{The isomorphism via Floer cohomology for Lagrangian fillings}

Let $L$ be an exact Lagrangian filling of $\Lambda$ with even Maslov number, $m(L)$, and let $h:L \rightarrow \R$ be a proper Morse function on $L$ that agrees with the projection $L\subset \mathit{Symp}(J^1M) = \R \times J^1M \rightarrow \R$ on $h^{-1}([T,+\infty))$ for $T \gg 0$.  In the case that $\mathbb{F} \neq 2$, assume that $L$ is equipped with a spin structure extending a given spin structure on $\Lambda$.
\begin{proposition} \label{prop:Seidel}  For any pair of local systems, $\rho_1,\rho_2: \pi_1(L, x_0) \rightarrow \mathbb{F}^*$, there is an isomorphism
\begin{equation} \label{eq:Seidel}
H^*Hom_+(\epsilon_{(L, \rho_1)}, \epsilon_{(L, \rho_2)}) \cong H^*(L,h; \rho_1^t \otimes \rho_2)
\end{equation}
that preserves grading mod $m(L)$.
\end{proposition}
Here, we view the $\rho_i$ as representations on the $1$-dimensional vector space $\mathbb{F}^1$.  Using the standard definitions\footnote{For representations $\sigma: G \rightarrow \mathit{GL}(V)$ and $\tau:G \rightarrow \mathit{GL}(W)$, with $V^\vee$ denoting the dual vector space, the dual representation $\sigma^t: G \rightarrow \mathit{GL}(V^\vee)$ has $\sigma^t(g) = (\sigma(g^{-1}))^\vee$ and $\sigma\otimes\tau:G \rightarrow GL(V \otimes W)$ has $(\sigma \otimes \tau)(g) = \sigma(g) \otimes \tau(g)$.} of tensor product and dual  for group representations, along with the canonical isomorphism of vector spaces $(\mathbb{F}^1)^\vee\otimes_\mathbb{F}\mathbb{F}^1= \mathbb{F}^1$ leads to 
\begin{equation} \label{eq:rho}
(\rho_1^t\otimes \rho_2)(g) = \rho_1(g)^{-1} \cdot \rho_2(g).
\end{equation}

\begin{proof}
The proof is based on the acyclicity of the version of the Lagrangian Floer complex for fillings used in \cite{Ekh, DR}, upgraded to incorporate local systems.  For a pair of transversally intersecting Lagrangian fillings, $L_1$ and $L_2$, of $\Lambda_1$ and $\Lambda_2$ (equipped with Maslov potentials for grading purposes) a $\Z/m\Z$-graded complex where $m=\mbox{gcd}(m(L_1),m(L_2))$ of the form
\[
CF^*(L_1,L_2) = C_\infty^*(L_1,L_2)\oplus C_0^*(L_1,L_2)
\]
is defined where $C_\infty^*(L_1,L_2)$ and $C_0^*(L_1,L_2)$ are modules freely generated by the $\Reeb_{1,2}$ Reeb chords of $\Lambda = \Lambda_1 \sqcup \Lambda_2$ and the points of intersection $L_1 \cap L_2$, respectively. A generalization where $CF^*(L_1,L_2)$ is a right $R[\pi_1(L_2)]$-module  appears in the work \cite[Section 8.2]{CDGG} (that also considers Lagrangian cobordisms with non-empty negative ends).  To also keep track of homotopy classes of boundary segments on $L_1$ one can define $CF^*(L_1,L_2; \pi_1)$ to be the free $(R[\pi_1(L_1)], R[\pi_1(L_2)])$-bimodule with the above generators, i.e. the free left $R[\pi_1(L_1)]\otimes_R R[\pi_1(L_2])^{\mathit{op}}$-module.  The differential has the triangular form
\[
\left[\begin{array}{cc} d_\infty & \phi \\ 0 & d_0 \end{array}\right] :C_\infty^*(L_1,L_2;\pi_1)\oplus C_0^*(L_1,L_2;\pi_1) \rightarrow C_\infty^*(L_1,L_2;\pi_1)\oplus C_0^*(L_1,L_2;\pi_1).
\]
The components are as follows:

\medskip

\noindent {\bf 1.}  The subcomplex $(C_\infty^*(L_1,L_2;\pi_1), d_\infty)$ can be viewed as the grading shifted linearized complex $C_{1,2}^{*-2}(f_{L_1},f_{L_2})$ (with grading collapsed mod $m$) for the $R[\pi_1(L_i)]$-valued augmentations $f_{L_i}$ of $\Lambda_i$ induced by the fillings $L_i$ from equation (\ref{eg:fL}).  Explicitly, it is 
\begin{align*}
C_\infty^*(L_1,L_2;\pi_1) & = \mathit{Span}_{R[\pi_1(L_1)]\otimes_R R[\pi_1(L_2)])^{\mathit{op}}} \Reeb_{1,2}(\Lambda_1 \sqcup \Lambda_2)
\end{align*}
where the grading of Reeb chords is $2$ larger than in the Legendrian contact DGA, with differential
\begin{equation} \label{eq:dinf}
d c = \sum \# \mathcal{M}(a; \mathbf{b}, c, \mathbf{d}) \, 
f_{L_1}(\mathbf{d}) \cdot  a \cdot f_{L_2}( \mathbf{b}).
\end{equation}

\medskip

\noindent {\bf 2.}  The map $\phi:C_0^*(L_1,L_2;\pi_1) \rightarrow C_\infty^{*+1}(L_1,L_2;\pi_1)$ is defined by 
\[
\phi(x) = \sum  \#\mathcal{M}(a; \alpha, x, \beta)\,  \beta \cdot a \cdot \alpha
\]
where the moduli space $\mathcal{M}(a; \alpha, x, \beta)$ consists of holomorphic disks  $u$ in $\mathit{Symp}(J^1M)$ with two boundary punctures: a positive puncture at $a \in \Reeb_{1,2}(\Lambda_1 \sqcup \Lambda_2)$, and a puncture limiting to a double point $x \in L_1\cap L_2$.  The boundary segments $u_{[a,x]}$ and $u_{[x,a]}$ are required to map to $L_2$ and $L_1$ and satisfy
\[
[(\gamma_a^+)^{-1}* u_{[a,x]}* \gamma_x^+] = \alpha \in \pi_1(L_2,x_2), \quad [(\gamma_x^-)^{-1} *  u_{[x,a]}*\gamma_a^-] = \beta \in \pi_1(L_1,x_1).
\]

\medskip

\noindent {\bf 3.}  The map $d_0:C_0^*(L_1,L_2;\pi_1) \rightarrow C_0^{*+1}(L_1,L_2;\pi_1)$ is
\[
d_0 y = \sum \#\mathcal{M}(x; \alpha,y,\beta) \, \beta \cdot x \cdot \alpha
\]
where $\mathcal{M}(x;\alpha,y,\beta)$ consists of disks with two punctures limiting to $x, y \in L_1 \cap L_2$, such that the counterclockwise boundary intervals $[x,y]$ and $[y,x]$ are mapped to $L_2$ and $L_1$ respectively with 
\[
[(\gamma_x^+)^{-1}* u_{[x,y]}* \gamma_y^+] = \alpha \in \pi_1(L_2,x_2),  \quad [(\gamma_y^-)^{-1} *  u_{[y,x]}*\gamma_x^-] = \beta \in \pi_1(L_1,x_1).
\] 

\medskip

The construction and invariance properties of $CF^*(L_1,L_2; \pi_1)$ as in \cite{Ekh,DR} go through with the same proofs, observing that, when verifying the required identities of maps on generators, in all cases broken disks that glue into the same $1$-dimensional moduli spaces contribute the same left and right $\pi_1$-coefficients. In particular, since the invariance statements allow  $L_1$ and $L_2$ to be modified in such a way as to remove all generators, we get the following;  see also \cite[Theorem 8.3]{CDGG} for this statement in a $\pi_1$-coefficient  setting very similar to the present one.  

\medskip

\noindent {\bf Fact:}  $CF^*(L_1,L_2; \pi_1)$ is acyclic.

\medskip
To verify (\ref{eq:Seidel}), given $L$ with $\rho_1$ and $\rho_2$, put $L_1=L$ and let $L_2=L^+$ be the perturbation of $L$ constructed from the Morse function $h:L \rightarrow \R$ together with $F:\Lambda \rightarrow \R$ as in \cite[Section 6.2]{DR}.  Use the representations $\rho_1$ and $\rho_2$ to change $CF^*(L_1,L_2;\pi_1)$ into an $\mathbb{F}$-module.  To do this, put $S:=R[\pi_1(L_1,x_1)] \otimes R[\pi_1(L_2,x_2)]^\mathit{op}$ and tensor
\begin{equation} \label{eq:CF}
CF^*((L_1, \rho_1),(L_2, \rho_2)) :=   \mathbb{F} \otimes_S CF^*(L_1,L_2;\pi_1) 
\end{equation}
where the right $S$-module structure on $\mathbb{F}$ is from the homomorphism
\[
S^{\mathit{op}} \rightarrow (\mathbb{F} \otimes_\mathbb{F} \mathbb{F}) \stackrel{\cong}{\rightarrow} \mathbb{F}= \mathit{End}(\mathbb{F}), \quad \beta \otimes \alpha \, \mapsto \rho_1(\beta) \otimes \rho_2(\alpha) \, \mapsto \,\rho_1(\beta) \cdot \rho_2(\alpha).
\]
 The result remains acyclic by the universal coefficient theorem, so that $\phi$ still provides a quasi-isomorphism
\[
\phi: C_0^*((L_1, \rho_1),(L_2, \rho_2)) \stackrel{\mbox{\small q.i.}}{\longrightarrow} C_\infty^{*+1}((L_1, \rho_1),(L_2, \rho_2)).
\] 
After tensoring, since (\ref{eq:dinf}) turns into (\ref{eq:lin})  the subcomplex $(C_\infty, d_\infty)$ becomes precisely  
\[
C_\infty^{*}((L_1, \rho_1),(L_2, \rho_2))= C_{1,2}^{*-2}(\epsilon_{(L,\rho_1)}, \epsilon_{(L,\rho_2)}) =\mathit{Hom}^{*-1}_+(\epsilon_{(L,\rho_1)}, \epsilon_{(L,\rho_2)}).
\]
  Moreover, as in \cite{Ekh} and \cite[Section 6.2]{DR}, the intersection points $L_1 \cap L_2$ are in grading preserving bijection with $\mathit{Crit}(h)$ and the holomorphic strips, $[u] \in \mathcal{M}(x; \alpha, y, \beta)$, used in the definition of $d_0$ are in (orientation sign preserving) bijection with ascending gradient trajectories of $h$ from $y$ to $x$, $[\eta] \in \mathcal{M}_h(x,y; \beta)$.  Under this bijection, the gradient trajectory, $\eta_u$, from $y$ to $x$ corresponding to a strip $u$ is homotopic to $u_{[y,x]}$. Hence, when $\eta_u^{-1}$ is concatenated with basepoint paths it becomes $\alpha \simeq \beta^{-1}$. Thus,
\[
d_0y = \sum \#\mathcal{M}(x; \alpha, y, \beta) \rho_1(\beta) \cdot \rho_2(\alpha) x = \sum \#\mathcal{M}_h(x, y;  \alpha) \rho_1(\alpha^{-1}) \cdot \rho_2(\alpha) x 
\]
is exactly the differential from the Morse complex $C^*(L, h; \rho_1^t \otimes \rho_2)$. (See \cite[Section 5.2]{EES07} for the comparison of coherent orientations of the moduli spaces of disks and gradient trajectories.)  In summary, we have
\[
C_0^{*}((L_1, \rho_1),(L_2, \rho_2)) = C^*(L, h; \rho_1^t \otimes \rho_2),
\] 
so that $\phi$ provides the required isomorphism (\ref{eq:Seidel}) completing the proof.
\end{proof}

\begin{remark}[Higher rank] 
The proof extends with a small adjustement to establish the isomorphism 
\[
H^*Hom_+(\phi_{(L, \rho_1)}, \phi_{(L, \rho_2)}) \cong H^*(L,h; \rho_1^t \otimes \rho_2)
\]
where $\rho_i:\pi_1(L, x_0) \rightarrow GL(V_i)$ are higher rank local systems and $\phi_{(L, \rho_i)} = \rho_i \circ f_L: \mathcal{A}(\Lambda) \rightarrow \mathit{End}(V_i)$ are the induced representations.  Indeed, in place of equation (\ref{eq:CF}), form
\[
CF^*((L_1, \rho_1),(L_2, \rho_2)) :=   (V_1^\vee \otimes_\mathbb{F} V_2) \otimes_S CF^*(L_1,L_2;\pi_1)
\]
where $V_1^\vee$ is the vector space dual to $V_1$ and $V_1^\vee \otimes_\mathbb{F} V_2$ is a right $S$-module via
\[
(\xi_1 \otimes v_2)(\beta \otimes \alpha) = (\rho_1(\beta)^\vee \xi_1) \otimes (\rho_2(\alpha) v_2).  
\]
The isomorphism then follows as above by computing the new $d_\infty$ and $d_0$ and comparing with (\ref{eq:LinHigher}) and (\ref{eq:MorseHigher}).
\end{remark}

At this point we can prove the Proposition \ref{prop:main2}.

\begin{proof}[Proof of Proposition \ref{prop:main2}]  That $f^*_L$ is algebraic is clear since $f_L$ is a ring homomorphism. 
To prove injectivity of $f_L^*$ under the assumption that $m(L)=0$, suppose that $\rho_1 \neq \rho_2$ are distinct local systems in $\mathit{Aug}(L,\mathbb{F})$.  From equation (\ref{eq:rho}) we see that $\rho^t_1 \otimes \rho_1$ is the trivial local system, but $\rho_1^t \otimes \rho_2$ is not.  Then, applying Propositions \ref{prop:Seidel} and \ref{prop:standard} we see that 
\[
H^{0}Hom_+(\epsilon_{(L, \rho_1)}, \epsilon_{(L, \rho_1)}) \not \cong H^{0}Hom_+(\epsilon_{(L, \rho_1)}, \epsilon_{(L, \rho_2)}).
\]
Certainly, this shows that $\epsilon_{(L, \rho_1)} \neq \epsilon_{(L, \rho_2)}$.  

In the case where $\dim L =2$, the Morse function $h:L \rightarrow \R$ can be chosen to only have critical points of index $0$ and $1$.  Thus, under the weaker assumption that $m(L)=m$ is even,  Proposition \ref{prop:Seidel} provides the isomorphism
\begin{align*}
H^0(L, h; \rho_1^t \otimes \rho_2) &= H^{\mathit{even}}(L, h; \rho_1^t \otimes \rho_2) \cong H^{\mathit{even}}Hom_+(\epsilon_{(L, \rho_1)}, \epsilon_{(L, \rho_2)}), 
\end{align*}
so that $\epsilon_{(L, \rho_1)}$ and $\epsilon_{(L, \rho_2)}$ are distinct when $\rho_1 \neq \rho_2$.  Moreover, $H^{\mathit{even}}Hom_+(\epsilon_{(L, \rho_1)}, \epsilon_{(L, \rho_1)}) \not \cong H^{\mathit{even}}Hom_+(\epsilon_{(L, \rho_1)}, \epsilon_{(L, \rho_2)})$  shows $\epsilon_{(L, \rho_1)}$ and $\epsilon_{(L, \rho_2)}$ are not isomorphic in the ($m$-graded version of the) positive augmentation category, $\mbox{Aug}_+(\Lambda, \mathbb{F})$, so \cite[Proposition 5.19]{NRSSZ} shows that they are not DGA homotopic either.
\end{proof}

\section{Augmentations without algebraic tori neighborhoods}  \label{sec:lambda}

This section contains examples of augmentations obstructed from having any orientable Lagrangian filling via Proposition \ref{prop:main2}.  First, we consider the specific case of $\Lambda_n$ in \ref{sec:lambdan}, and then observe the extension to general negative twist knots in \ref{sec:general}. In \ref{sec:formality} we discuss in more detail the  Ekholm-Lekili obstruction to fillings arising from $A_\infty$ structures as in \cite{EL,Etgu} and establish that it does not apply to the augmentations of $\Lambda_n$ that we consider.

\subsection{Augmentations of $\Lambda_n$} \label{sec:lambdan}
Recall now the family of Legendrian knots, $\Lambda_{n}$, presented in Figure \ref{fig:Lambda} where $n =2k+1\geq 3$ is an odd integer.  Using the Ng resolution procedure \cite{Ng}, the Reeb chords of $\Lambda_n$ are: 
\begin{center}
\renewcommand{\arraystretch}{1.2} 
\begin{tabular}{c|c|c}

Degree & $0$ & $1$ \\
\hline 
Reeb Chords & $a,b, c_1, \ldots, c_n$ & $e_0, e_1, \ldots, e_n$
\end{tabular}
\end{center} 
Let $t$ denote the generator of   $\Z[\pi_1(\Lambda_n, x_0)]$, and choose all basepoint paths  to avoid the right cusp $e_0$.  Using the sign convention from \cite[Section 2.1]{LeRu}, the differentials in $\mathcal{A}(\Lambda_n)$ are
\begin{align}
\label{eq:1} \partial e_0 & = t^{-1} +c_n(1+ab) \\
\label{eq:2} \partial e_1 & = 1 + (1+ba)(-c_1) + \left\{\begin{array}{cr} (-b)(1+c_2c_3), & n=3  \\ 0, & n>3 \end{array} \right.\\
\label{eq:3} \partial e_j & = 1+c_{j-1}c_j, \quad \quad \quad 2\leq j \leq k+1 \\
\label{eq:4} \partial e_j & = 1+c_{j}c_{j-1}, \quad \quad \quad k+2 \leq j \leq n. 
\end{align}
Note that since all Reeb chords have degree $0$ or $1$, any even graded augmentation is automatically $0$-graded.  I.e., for any even $m$, the $m$-graded augmentation variety has
\[
\mathit{Aug}_m(\Lambda_n,\mathbb{F}) = \mathit{Aug}(\Lambda_n,\mathbb{F}).  
\]  

\begin{proposition} \label{prop:variety}
For any odd $n\geq 3$ and any field $\mathbb{F}$, there is an algebraic isomorphism
\[
\mathit{Aug}(\Lambda_n,\mathbb{F}) \cong V:= \{(a,b)\in \mathbb{F}^{2}\,|\,  ab \neq -1\}.
\]
Moreover, for any $\epsilon \in \mathit{Aug}(\Lambda_n, \mathbb{F})$ the linearized cohomology ring, $H^*Hom_+(\epsilon,\epsilon)$, is isomorphic  to $H^*(T^2-\{y_0\}; \mathbb{F})$ of a punctured torus.   
\end{proposition}

\begin{proof}
As explicit algebraic sets,
\[
\mathit{Aug}(\Lambda_n,\mathbb{F}) = \{(a,b,c_1, \ldots, c_n,t) \in \mathbb{F}^{n+2}\times \mathbb{F}^*\,|\, \partial e_j = 0, \, 0 \leq j \leq n\} 
\]
 and
\[
V=\{(a,b,c) \in \mathbb{F}^3\,|\, (ab+1)c =1\}.
\]
Solving the equations $\partial e_j = 0$ using (\ref{eq:1})-(\ref{eq:4}) one sees that the projection to the first three coordinates is an algebraic bijection 
\[
\begin{array}{ccc}
\mathit{Aug}(\Lambda_n,\mathbb{F}) & \stackrel{\cong}{\rightarrow} & V \\
(a,b,c_1, \ldots, c_n,t) & \mapsto & (a,b, c_1)
\end{array}
\]
with algebraic inverse
\begin{align*}
(a,b, c) & \mapsto (a,b, c_1, \ldots, c_n,t)  \\
c_j &= \left\{\begin{array}{cr} c & \mbox{$j$ odd}, \\ -(ab+1) & \mbox{$j$ even} \end{array} \right. \\
t &= -1.
\end{align*}

To verify the statement about $H^*Hom_+(\epsilon,\epsilon)$, note that by the Sabloff Duality Theorem, \cite{Sab}, the ordinary linearized homology of $\Lambda_n$ must have $\mathit{dim}_\mathbb{F}LCH^\e_1(\Lambda_n) = 1$ (since $\mathit{dim}_\mathbb{F}LCH^\e_1(\Lambda_n) = \mathit{dim}_\mathbb{F}LCH^\e_{-1}(\Lambda_n)+1$ and there are no Reeb chords in degree $-1$).  As the Euler characteristic is $\mathit{tb}(\Lambda_n) = 1$, we have $\mathit{dim}_\mathbb{F}LCH^\e_0(\Lambda_n) = 2$.  Now, using the isomorphism $H^*\mathit{Hom}_+(\epsilon, \epsilon) \cong LCH^\e_{1-*}(\Lambda)$ from \cite[Corollary 5.6]{NRSSZ} we have $\dim_\F H^0\mathit{Hom}_+(\epsilon, \epsilon)=1$ and  $\dim_\F H^1\mathit{Hom}_+(\epsilon, \epsilon)=2$ agreeing with $H^*(T^2-\{y_0\};\mathbb{F})$.  Moreover, since both $H^*\mathit{Hom}_+(\epsilon, \epsilon)$ and $H^*(T^2-\{y_0\}; \mathbb{F})$ have an identity element in degree $0$ the grading determines the ring structures.  
\end{proof}

\begin{remark}  

Obstructions (1)-(3) discussed in the introduction do not apply to $\epsilon_3$.

\begin{enumerate}
\item $\mathit{tb}(\Lambda_n) =1$, so a genus $1$ orientable filling $L$ would be allowed by (1) since it would have $2g(L)-1=1$.

\item $LCH^{\e_3}_{*}(\Lambda) \cong H^{1-*}\mathit{Hom}_+(\epsilon_3, \epsilon_3) \cong H^{1-*}(T^2-\{y_0\}; \mathbb{F}) \cong H_{1+*}(T^2-B^2, \partial( T^2-B^2); \mathbb{F})$.

\item The $A_\infty$-algebra $\text{RHom}_{CE^*}(\epsilon_3,\epsilon_3)$ is isomorphic to the cochain complex of the torus. See section \ref{sec:formality} for the notation and details. Hence the Ekholm-Lekili obstruction does not apply.

\end{enumerate}

\end{remark}

\begin{proposition} \label{prop:V}
Let $k$ be an algebraically closed field, and $V = \{(a,b) \in k \,|\, ab \neq -1\}$.  There is no injective algebraic map 
\[
\phi:(k^*)^2 \rightarrow V
\]
having $(0,0)$ in its image.
\end{proposition}

\begin{proof}
Suppose $\phi$ is injective with $\phi(s_0,t_0) = 0$, and view 
\[
V = \{(a,b,c) \in k^3\,|\, (ab+1)c =1\}.
\]
The homomorphism of coordinate rings associated to $\phi$ is a ring homomorphism
\[
\phi^*:k[a,b,c]/\left( (ab+1)c-1\right) \rightarrow k[s^{\pm1},t^{\pm1}]. 
\]
Writing
\[
A(s,t) = \phi^*a, \quad B(s,t) = \phi^*b, \quad C(s,t) = \phi^*c,
\]
we have 
\[
(AB+1)C =1.
\]
Thus, $AB+1$ is a unit in $k[s^{\pm1},t^{\pm1}]$ and hence has the form $\alpha s^mt^n$ for some $\alpha \in k^*$ and $m,n \in \Z$ so that  
\begin{equation} \label{eq:ABequation}
A(s,t) B(s,t) = \alpha s^mt^n -1.
\end{equation}
Moreover, since $\phi(s_0,t_0) = (0,0)$, we get  $A(s_0,t_0) = B(s_0,t_0)=0$.  By modifying $A$ and $B$ via a substitution of the form $s \mapsto s^{\pm1}$ and $t \mapsto t^{\pm1}$, we can assume $m,n \geq 0$.

\medskip

\noindent{\bf Case 1.}  $m=n = 0$.

\medskip

We must have $\alpha = 1$, since otherwise $AB$ would not have any zero.  Thus, $AB = 0$ which implies that one of $A$ or $B$ is $0$.  Supposing $A=0$, we have
\[
\phi(s,t) = (0, B(s,t)),
\]
and this contradicts the injectivity of $\phi$.  
\medskip

\noindent{\bf Case 2.}  $m>0$ or $n>0$, i.e. $\alpha s^mt^n -1 \in k[s,t]$ is non-constant.

\medskip

To work in $k[s,t]$ rather than $k[s^{\pm1},t^{\pm1}]$,  we put
\[
\widetilde{A} = s^{-m(A,s)}t^{-m(A,t)}A, \quad \widetilde{B} = s^{-m(B,s)}t^{-m(B,t)}B
\]
where $m(A,s)$ and $m(A,t)$ denote the minimum degree of $A$ in $s$ and $t$ respectively, and similarly for $B$. Note that $m(A,s) + m(B,s) = 0$, because the right hand side of (\ref{eq:ABequation}) regarded as a Laurent polynomial in $s$ has minimum degree $0$. Similarly $m(A,t) + m(B,t) = 0$.
We then have  
\[
\widetilde{A}\,\widetilde{B}=\alpha s^mt^n -1, \quad \widetilde{A}, \widetilde{B} \in k[s,t],
\]
and the zero sets satisfy 
\[
Z(\widetilde{A}) \cap (k^*)^2 = Z(A) \cap (k^*)^2 \quad \mbox{and} \quad  Z(\widetilde{B}) \cap (k^*)^2 = Z(B) \cap (k^*)^2. 
\]
By the injectivity of $\phi$, we moreover have
\begin{equation} \label{eq:Z}
Z(\widetilde{A}) \cap Z(\widetilde{B}) \cap (k^*)^2 = \phi^{-1}(\{(0,0)\}) = \{(s_0,t_0)\}.
\end{equation} 

\medskip

\noindent {\bf SubCase 2A.}  $\Char k = 0$. One of $m$ or $n$ is strictly positive; without loss of generality, suppose $m>0$.  We have
\[
\widetilde{A}(s,t_0) \widetilde{B}(s,t_0) = (\alpha t_0^n)s^m  -1,
\]
but this is a contradiction since $s_0$ would be a multiple root and $(\alpha t_0^n)s^m  -1 \in k[s]$ is seperable, i.e. has no multiple roots, since it is relatively prime to its formal derivative:
\[
(-1)((\alpha t_0^n)s^m  -1) + \left(\frac{s}{m}\right)((m \alpha t_0^n)s^{m-1}) = 1.
\]

\medskip

\noindent {\bf SubCase 2B.}  $\Char k= p >0$.  Take $\ell \geq 0$ such that $m = p^\ell m'$ and $n = p^\ell n'$ with $m', n' \in \Z$ and either $(p, m') =1$ or $(p, n') = 1$, and let $\alpha' \in k$ have $(\alpha')^{p^\ell} = \alpha$.  With no loss of generality assume $(p,m') =1$.  Then, using that $k[s,t]$ is UFD we write
\begin{align*}
\widetilde{A} \, \widetilde{B} \, =\, (\alpha's^{m'}t^{n'} -1)^{p^\ell} \,=\, (q_1q_2\ldots, q_r)^{p^\ell}
\end{align*}
where $q_1, \ldots, q_r$ are the irreducible factors of $\alpha's^{m'}t^{n'} -1$ in $k[s,t]$.    

\medskip

\noindent{\bf Claim.}  No two of the $q_i$ and $q_j$ with $i\neq j$ can have a common zero in $k^2$.

\medskip

To verify the claim indirectly, assume that $(s_1,t_1)$ is a common zero of $q_i$ and $q_j$.  Then, $s_1$ is a root of multiplicity $\geq 2$ for 
\[
g(s) = (\alpha' t_1^{n'})s^{m'} -1 = q_1(s,t_1) \cdots q_r(s,t_1) \in k[s],
\]
 contradicting that $g(s)$ is seperable.  [Since $(m',p) =1$, the same formal derivative computation as above goes through.]

Now, $\widetilde{A}$ and $\widetilde{B}$ do have the common zero $(s_0,t_0) \in (k^*)^2$, so the claim shows there must be some $q_i$ such that
\[
q_i \,|\, \widetilde{A},  \quad  \quad q_i\,|\, \widetilde{B}, \quad \mbox{and} \quad q_i(s_0,t_0) = 0.
\]
In view of (\ref{eq:Z}), this $q_i \in k[s,t]$ is then an irreducible polynomial with
\[
\{(s_0,t_0)\}\subset Z(q_i) \subset (k\times\{0\})\cup (\{0\}\times k) \cup  \{(s_0,t_0)\}.
\]
No such polynomial can exist.  [To see this, note that $Z(q_i) \cap (k\times\{0\})$ is a Zariski closed subset of $k\times\{0\}$ and hence is finite or $k\times\{0\}$.  It cannot be all of $k\times\{0\}$ since in this case $s \,|\, q_i$, and as $q_i$ is irreducible we would have $q_i = \beta s$ for some $\beta \in k^*$.  This would imply $q_i(s_0,t_0) \neq 0$, a contradiction.  Similarly, $Z(q_i) \cap (\{0\} \times k)$ is finite.  Thus, $Z(q_i)$ itself must be finite, but this is impossible since $q_i$ is non-constant and $k$ is algebraically closed.]  
\end{proof}

\begin{corollary} \label{cor:main} Let $n\geq 3$ be odd.
\begin{enumerate}
\item  There is no orientable Lagrangian filling inducing the augmentation $\epsilon_3 \in \mathit{Aug}(\Lambda_n, \mathbb{F}_2)$ with $\epsilon_3(a) = \epsilon_3(b) = 0$.
\item  There is no pair $(\Lambda, \rho)$ consisting of an orientable Lagrangian filling with local system $\rho:\pi(\Lambda) \rightarrow \mathbb{F}^*$ that induces the augmentation $\epsilon \in \mathit{Aug}(\Lambda_n, \mathbb{F})$ with $\epsilon(a) = \epsilon(b) = 0$.  
\end{enumerate}
\end{corollary}
\begin{proof}
If such a $(L, \rho)$ exists, then it also  induces $\e$ in $\mathit{Aug}(\Lambda_n, \overline{\mathbb{F}})$ where $\overline{\mathbb{F}}$ denotes the algebraic closure of $\mathbb{F}$.  [Just extend the codomain of $\rho$ via the inclusion $\mathbb{F} \subset \overline{\mathbb{F}}$.]  Since $1= \mathit{tb}(\Lambda_n) = 2g(L) -1$, we must have $\dim H_1(L) = 2$.  Moreover, since $L$ is orientable $m=m(L)$ is even, and Proposition \ref{prop:main2} shows $(0,0)\in \mathit{Aug}_m(\Lambda_n, \overline{\mathbb{F}}) =\mathit{Aug}(\Lambda_n, \overline{\mathbb{F}}) \cong V$ would then be in  the image of the injective algebraic map  $f_L^*:(\overline{\mathbb{F}}^*)^2 \rightarrow V$.  This contradicts Proposition \ref{prop:V}. 
\end{proof}

\begin{remark}
In contrast, the augmentations $\epsilon_1, \epsilon_2 \in \mathit{Aug}(\Lambda_n, \F_2)$ from the introduction can be induced by Lagrangian fillings with $m(L) = 0$.  See Figure \ref{fig:Filling1} which also demonstrates an immersed Lagrangian disk filling with a single double point that contains $\epsilon_3$ in its induced augmentation set. 
\end{remark}

\begin{figure}[!ht]
\labellist
\small
\pinlabel $L_1$ [t] at 90 -2
\pinlabel $L_2$ [t] at 352 -2
\pinlabel $L_3$ [t] at 614 -2

\endlabellist
\includegraphics[scale=.6]{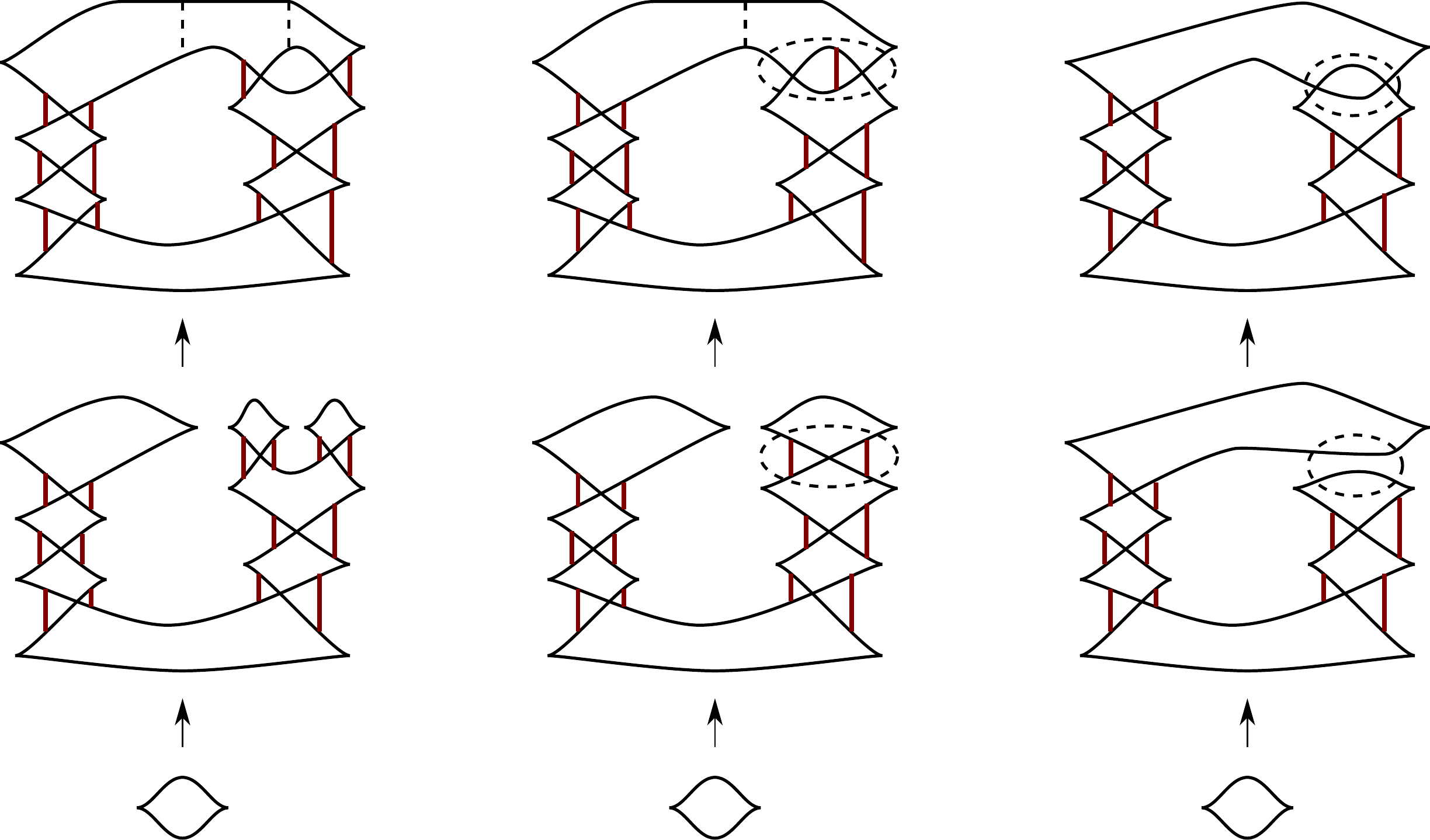}
\vspace{0.7cm}
\caption{Lagrangian fillings $L_1, L_2,L_3$ inducing the augmentations $\e_1, \e_2, \e_3$ respectively.  The dotted vertical lines indicate the location of pinch moves.  For $L_2$ a $D_4^-$ move occurs in the circled location.  For $L_3$ a clasp move occurs in the circled location, resulting in a double point of the filling.  The induced augmentations may be computed using the method of \cite[Section 6]{PanRu2}, and the computation is indicated by the solid red vertical lines that form the handleslides for a $2$-dimensional Morse complex family for each of the fillings.  Those crossings, $x$, of $\Lambda_n$ that end up with a handleslide positioned immediately to their left  have $\epsilon(x) =1$.  Alternatively, one can use \cite{EHK} to compute the augmentations induced by $L_1$ and $L_2$.}
\label{fig:Filling1}
\end{figure}

\subsection{Formality}\label{sec:formality}
We show in Proposition \ref{ELformal} that in characteristic $2$ the $A_\infty$ algebra $B = \mathrm{RHom}_{CE^*}(\epsilon_3,\epsilon_3)$ (this is denoted by $\mbox{RHom}_{CE^*}(\mathbb{K},\mathbb{K})$ in \cite{EL,Etgu}) is formal, and isomorphic to the cochain complex of the torus. As discussed in \cite[Section 3]{Etgu} a consequence of \cite[Theorem 4]{EL} is that, when $\e$ is induced by a Lagrangian filling, $\mathrm{RHom}_{CE^*}(\epsilon,\epsilon)$ is quasi-isomorphic as an $A_\infty$ algebra to the cochain complex of the filling, capped off by a disk. Therefore, Proposition \ref{ELformal} shows that this obstruction of Ekholm-Lekili  does not apply to $\e_3$. In our setup, the strictly unital $A_\infty$ algebra $B$ can be obtained from the non-unital $A_\infty$ algebra, $\mathit{LCH}^{*-1}_{\e_3}(\Lambda_n) =\mathit{Hom}^{*}_-(\e_3,\e_3)$ defined in \cite{CEKSW} (that is also the hom-space in the negative augmentation category from \cite{BC}), by adding a copy of $\mathbb{F}$ to make it unital \cite[Remark 5]{EL}.

Let $B' =\mathit{Hom}_-(\e_3,\e_3)$, and let $\Char \mathbb{F}=2$. We consider $\Lambda_n$ with $n\geq 5$ odd. Abusing the notations between Reeb chord generators and their duals, we have $B' =\mathit{Span}_{\mathbb{F}}\{a,b, e_0, c_i, e_i\}$, $1\leq i\leq n$. The degree of generators in $\mathit{Hom}_-(\e_3,\e_3)$ is one larger than their degree in the Legendrian contact DGA. Recall $n= 2k+1$. Following the recipe of \cite{CEKSW}, the $A_\infty$ algebra structure on $B'$ is given by:

\begin{align*}
\mu^1_{B'}(c_i) & = e_i + e_{i+1}, \quad 1\leq i\leq n-1, &
\mu^1_{B'}(c_n) & = e_n + e_0, \\
\mu^2_{B'}(c_i,c_{i+1}) & = e_{i+1}, \quad\qquad 1\leq i\leq k, &
\mu^2_{B'}(b,a) &= e_1, \\
\mu^2_{B'}(c_i,c_{i-1}) & = e_{i}, \qquad\qquad k+2\leq i\leq n,  &
\mu^2_{B'}(a,b) &= e_0, \\
\mu^3_{B'}(b,a,c_1) &= e_1, &
\mu^3_{B'}(c_n,a,b) &= e_0.
\end{align*}
\smallskip

Let $B = \mathbb{F}\oplus B'$, and $1\in \mathbb{F}$ be the unit in $B$. The $A_\infty$ operations extend to $B$ in the canonical way, namely $\mu^1_B(1) = 0$, $\mu^2_B(1,x) = \mu^2_B(x,1) =x$, and for $d\geq 3$, $\mu_B^{d}=0$ if the input contains $1$. In particular, $\mu_B^d = \mu_{B'}^d$ when inputs are from $B'\subset B$.

We use the homological perturbation lemma, see eg. \cite[Proposition 1.12]{Sei}, to find a minimal model for $B$.  Define new variables by $e_i' := e_i+e_0, c_i' := \sum_{j=i}^n c_j$ for $1\leq i \leq n$. Define sub-chain complexes of $B$:
$$A :=\mathit{Span}_{\mathbb{F}}\{ 1, a, b, e_0 \}, \qquad C :=\mathit{Span}_{\mathbb{F}}\{ e_i', c_i' \} =\mathit{Span}_{\mathbb{F}}\{ e_i', c_i \}.$$
Then $B = A\oplus C$. Let $F^1: A\rightarrow B$ be the embedding and $G^1: B \rightarrow A$ be the projection. Note that $A$ and $B$ are quasi-isomorphic. Define a degree $-1$ map $T^1: B\rightarrow B$, which is the extension by zero of $T^1: C\rightarrow C$, where
\begin{equation*}
T^1(e'_i) = c_i', \qquad T^1 (c_1)=0.
\end{equation*}
It is straightforward to verify the following:
\begin{equation}\label{hpl}
T^1\mu^1 + \mu^1T^1 = F^1G^1 + \textrm{id}_B,
\end{equation}
because for any element of $A$, both sides of \eqref{hpl} are zero, and for any element of $C$, both sides of \eqref{hpl} are the identity. The homological perturbation lemma implies that $A$ admits an $A_\infty$ algebra structure induced from $B$.

We compute this induced $A_\infty$ structure on $A$. Let $A' = \mathit{Span}_{\mathbb{F}}\{a, b, e_0 \}$, then $A = \mathbb{F}\oplus A'$. It suffices to compute the $A_\infty$ operations with inputs from $A'$. The unit $1$ can be invoked canonically. We still denote by $F^1: A' \rightarrow B'$ the inclusion and $G^1: B'\rightarrow C$ the projection. Recall the definitions of $F^d$ and $\mu^d_{A'}$ from \cite[(1.18)]{Sei}.

When $d=1$, $\mu^1_{A'} = 0$, and $F^1$ is the embedding. For $d=2$, there are
\begin{align*}
F^2(x,y) = T^1(\mu^2_{B'}(F^1(x),F^1(y))) = T^1(\mu^2_{B'}(x,y)),\\
\mu^2_{A'}(x,y) = G^1(\mu^2_{B'}(F^1(x), F^1(y))) = G^1(\mu^2_{B'}(x,y)).
\end{align*}
Note $\mu^2_{B'}$ vanishes on $A'\otimes A'$ except when $(x,y) = (a,b)$ and $(b,a)$, the only possible non-zero terms in $F^2$ and $\mu^2_{A'}$ are: 
\begin{align*}
F^2(a,b) &= T^1(\mu^2_{B'}(a,b)) = T^1(e_0) = 0, \\
\mu^2_{A'}(a,b) &= G^1(\mu^2_{B'}(a,b)) = G^1(e_0) = e_0, \\
F^2(b,a) &= T^1(\mu^2_{B'}(b,a)) = T^1(e_1) = T^1(e_0+e_1') = c_1', \\
\mu^2_{A'}(b,a) &= G^1(\mu^2_{B'}(b,a)) = G^1(e_1) = G^1(e_0+e_1') = e_0.
\end{align*}

\begin{lemma}\label{lem1}
If $d\geq 3$ is odd, then $F^d= 0$, $\mu^d_{A'}=0$.
\end{lemma}
\begin{proof}
Observe the following facts:
\begin{enumerate}
\item $\mu^d_{B'} = 0$ for $d\geq 4$.
\item $\mathrm{im}\, F^1\in A'$, $\text{im } F^d \subset C$ for $d\geq 2$.
\item $\mu_{B'}^3$ in non vanishing only on $A'\otimes A'\otimes C$ and $C\otimes A' \otimes A'$.
\item $\mu_{B'}^2$ vanishes on $A'\otimes C$ and $C\otimes A'$.
\end{enumerate}
We prove by induction. For $d = 3$, there are
\begin{align*}
F^3(x,y,z) &= T^1(\mu^2_{B'}(F^2(x,y),F^1(z))) + T^1(\mu^2_{B'}(F^1(x),F^2(y,z))) + T^1(\mu^3_{B'}(F^1(x),F^1(y),F^1(z))), \\
\mu^3_{A'}(x,y,z) &= G^1(\mu^2_{B'}(F^2(x,y),F^1(z))) + G^1(\mu^2_{B'}(F^1(x),F^2(y,z))) + G^1(\mu^3_{B'}(F^1(x),F^1(y),F^1(z))).
\end{align*}
Note the only non-trivial terms in $F^2$ is $F^2(b,a) = c_1' \in C$, but $\mu^2_{B'}$ vanishes on $A'\otimes C$ and $C\otimes A'$. Therefore $\mu^2_{B'}(F^2(x,y),F^1(z)) = \mu^2_{B'}(F^1(x),F^2(y,z)) = 0$ for any $(x,y,z) \in (A')^{\otimes 3}$.  Also $\mu^3_{B'}$ vanishes on $(A')^{\otimes 3}$. Therefore $F^3 =0, \mu^3_{A'}=0.$

For $d\geq 5$, consider the term $\mu_{B'}^2(F^i,F^{d-i})$. If $i=1$ or $d-i=1$, then it vanishes because of (4). Otherwise $i\neq 1,d-i\neq 1$, then one of them, say $F^i$, must be an odd number greater than $1$, by induction, $F^i =0$, so as the term $\mu^2_{B'}$. Next consider $\mu_{B'}^3(F^i,F^j, F^k)$. By $(3)$, we must have $i=j=1$, yielding $k$ an odd number, or $j=k=1$, yielding $i$ an odd number. In any case, the term $\mu^3_{B'}=0$. All higher $\mu^4_{B'}=0$ are zero by $(1)$. The proof is completed.
\end{proof}

\begin{lemma}
If $d\geq 4$ is even, then the only non-vanishing terms are:
\begin{align*}
F^d(b,a,\dotsb,b,a,b,a) & = c_1 + c_3 + \dotsb + c_n, \\
\mu^d_{A'}(b,a,\dotsb,b,a,b,a) &= e_0, \\
\mu^d_{A'}(b,a,\dotsb,b,a,a,b) &= e_0.
\end{align*}
\end{lemma}
\begin{proof}
We prove by induction. When $d =4$, since $F^3=0$, $\mu_{B'}^3 =0$ on $A'\otimes C\otimes A'$, and $\mu_{B'}^4= 0$ on $(A')^{\otimes 4}$, the possible non-vanishing terms are:
$$\mu_{B'}^2(F^2(x,y), F^2(z,w)),\quad \mu_{B'}^3(F^1(x), F^1(y), F^2(z,w)), \quad \mu_{B'}^3(F^2(x,y), F^1(z), F^1(w)).$$

The first two terms are non-zero iff $(x,y,z,w) = (b,a,b,a)$, the third term is non-zero iff  $(x,y,z,w) = (b,a,a,b)$.

If $(x,y,z,w) = (b,a,b,a)$, then the three terms above are $e_2 + e_3 + \dotsb +e_n$, $e_1$, $0$. Hence
\begin{align*}
F^4(b,a,b,a) &= T^1(e_1+ \dotsb + e_n ) = c_1' + c_2' + \dotsb + c_n' = c_1 + c_3 + \dotsb + c_n, \\
\mu^4_{A'}(b,a,b,a) &= G^1(e_1+ \dotsb + e_n) = e_0.
\end{align*}

If $(x,y,z,w) = (b,a,a,b)$, then three terms above are $0$, $0$, $e_0$. Hence
\begin{align*}
F^4(b,a,a,b) &= T^1(e_0) = 0, \\
\mu^4_{A'}(b,a,a,b) &= G^1(e_0) = e_0.
\end{align*}

For $d\geq 6$, consider $\mu_{B'}^2(F^i, F^{d-i})$. If $i$ and $d-i$ are odd, then apply Lemma \ref{lem1}. If $i, d-i$ are even numbers that are not $2$, then by induction hypothesis
$\mbox{im}\,\mu^2_{B'}(F^i,F^{d-i}) \subset \mathit{Span}_\mathbb{F} \{\mu^2_{B'}(c_1 + c_3 + \dotsb + c_n, c_1 + c_3 + \dotsb + c_n)\} = \{0\}$. For a non-vanishing $\mu_{B'}^3(F^i, F^j, F^k)$, it must be either $i=j=1$ or $j=k=1$. Summarizing, non-vanishing terms are:
$$ \mu_{B'}^2(F^{d-2}, F^2),\quad \mu_{B'}^2(F^2, F^{d-2}),\quad \mu_{B'}^3(F^1, F^1, F^{d-2}), \quad \mu_{B'}^2(F^{d-2}, F^1,F^1).$$
The first three terms are non-zero only on $(b,a,\dotsb,b,a,b,a)$, and the last term are non-zero only on $(b,a,\dotsb, b,a,a,b)$.

For $(b,a,\dotsb,b,a,b,a)$, then the four terms become $\displaystyle { \sum_{\textrm{even } i=2}^{k+1} e_i + \sum_{\textrm{odd } j= k+2}^n e_j }$, $\displaystyle { \sum_{\textrm{odd } i=2}^{k+1} e_i + \sum_{\textrm{even } j= k+2}^n e_j }$, $e_1$, $0$. Hence,
\begin{align*}
F^d(b,a,\dotsb,b,a,b,a) &= T^1(e_1+ \dotsb + e_n) = c_1' + c_2' + \dotsb + c_n' = c_1 + c_3 + \dotsb + c_n, \\
\mu^d_{A'}(b,a,\dotsb,b,a,b,a) &= G^1(e_1+ \dotsb + e_n) = e_0.
\end{align*}

For $(b,a,\dotsb, b,a,a,b)$, then the four terms become $0$, $0$, $0$, $e_0$. Hence
\begin{align*}
F^d(b,a,\dotsb, b,a,a,b) &= T^1(e_0) = 0, \\
\mu^d_{A'}(b,a,\dotsb, b,a,a,b) &= G^1(e_0) = e_0.
\end{align*}
\end{proof}

Summarizing, we have the following.
\begin{proposition}\label{minmodel}
Let $A =\mathit{Span}_{\mathbb{F}}\{ 1, a, b, e_0 \}$ be the strict unital $A_\infty$ algebra, where the non-trivial $A_\infty$ operations, except for those canonical ones involving the unit $1$, are
\begin{align*}
\mu^{2l}((b,a)^{l-1}, b,a) &= e_0, \\
\mu^{2l}((b,a)^{l-1}, a,b) &= e_0.
\end{align*}
Then $A\cong \mathrm{RHom}_{CE^*}(\epsilon_3, \epsilon_3)$ as $A_\infty$ algebras.
\end{proposition}

\begin{proposition}\label{ELformal}
$A$ is formal. By Proposition \ref{minmodel}, $\mathrm{RHom}_{CE^*}(\epsilon_3, \epsilon_3)$ is also formal.
\end{proposition}
\begin{proof}
Let $D$ be the formal $A_\infty$ algebra such that $D = A$ as cochain complexes and $\mu^2_D = \mu^2_A$. In particular, $\mu^d_D = 0$ for all $d \geq 3$. Define an $A_\infty$ algebra morphism $\{F^d: A^{\otimes d}\rightarrow D\}$, where the only non-vanishing terms are:
$$F^1 = \mathrm{id}, \qquad F^3(b,a,e_0) = e_0.$$
We check $\{F^d\}$ is an $A_\infty$ algebra morphism. Recall from \cite[(1.6)]{Sei} for the equations:
\begin{align*}
&\sum_r \sum_{s_1,\dotsb,s_r} \mu_D^r(F^{s_r}(a_d,\dotsb, a_{d-s_r+1}),\dotsb, F^{s_1}(a_{s_1},\dotsb, a_{1})) \\
=&\sum_{m,n}F^{d-m+1} (a_d,\dotsb,a_{n+m+1},\mu_A^{m}(a_{n+m}, \dotsb, a_{n+1}),a_n,\dotsb, a_1).
\end{align*}

The case $d=2$ can be checked easily. 

For $d\geq 3$, note (1) $\mu^{d}_D = 0$ for $d \neq 2$, (2) $\mathrm{im}\, F^d\subset \langle e_0 \rangle$ for $d\geq 2$, (3) $\mu^2_D(e_0, e_0) = 0$.
Then LHS vanishes for all $d$ except for the following two cases when $d=4$:
$$\mu^2_D(F^3(b,a,e_0),F^1(1)) = e_0,\qquad \mu^2_D(F^1(1), F^3(b,a,e_0)) = e_0$$

Suppose $d$ is odd. If $m$ is odd, then $\mu_A^m=0$, and RHS is zero; otherwise if $m$ is even, then $d-m+1$ is even, then $F^{d-m+1} =0$, the RHS is also zero.

Suppose $d$ is even. Note $F^{d-m+1}$ is only non zero for $d-m+1 =1$ or $3$. 
\begin{itemize}
\item When $d-m+1 = 1$, $F^1(\mu_A^d(\dotsb))\neq 0$ iff $\mu_A^d \neq 0$. Hence we must have $d$ even and input $(b,a)^l$ or $((b,a)^{l-1},a,b)$, where $d=2l$.
\item When $d-m+1 = 3$, observe that (1) $\mu_A^m =0$ for $m$ odd, (2) when the input contains $1$, then we must have $m=2$ for the output to be non-zero. The only input that may contribute a non-zero out put is $(1,b,a,e_0)$, $(b,1,a,e_0)$, $(b,a,1, e_0)$, $(b,a,e_0,1)$, $(b,a)^l$, or $((b,a)^{l-1}, a,b)$, where $d=2l$.
\end{itemize}

For $(1,b,a,e_0)$ and $(b,a,e_0,1)$, $\mathrm{LHS}=\mathrm{RHS}=e_0$. For $(b,1,a,e_0)$, $(b,a,1, e_0)$, $(b,a)^l$, or $((b,a)^{l-1}, a,b)$, $\mathrm{LHS}=\mathrm{RHS}=0$. Therefore, $\{F^d\}$ is an $A_\infty$ algebra morphism. That it is a quasi-isomorphism is clear. We have thus proved the formality.

\end{proof}

\subsection{Negative Legendrian twist knots: The general case}  \label{sec:general}

The classification of Legendrian twist knots appears in \cite{ENV}.  The Legendrians $\Lambda_n$ with $n \geq 3$ odd that we have considered belong to the topological twist knot type that is notated in \cite{ENV} as $K_{-n-1}$.  Note that $K_{-2}$ is a left handed trefoil and $K_{-1}$ is the unknot. The negative twist knots, $K_{-m}$ with $m\geq 3$ odd, have $\mbox{max-tb}(K_{-m}) = -3$, so no orientable fillings can exist.  

To recall the classification in knot types of $K_{-n-1}$ with $n=2k+1\geq 3$ odd, consider a Legendrian knot as pictured in Figure \ref{fig:TwistKnot} where the box contains a product, $X$, of a total of $n-1=2k$ tangles  of the form $Z^+, Z^-, S^+,S^-$.  Moreover, we require the orientation is such that from left to right the signs alternate as $+,-,+,-,\ldots$;  all of the $Z^+$'s appear to the left of all of the $S^+$'s; and all of the $Z^-$'s appear to the left of all of the $S^-$'s.  Let $0 \leq z_-,z_+ \leq k$ denote the number of $Z^+$'s and the number of $Z^-$'s that appear in $K_{z_+,z_-}$.  For example, $\Lambda_n = K_{\left\lceil \frac{k}{2}\right\rceil, \left\lfloor \frac{k}{2} \right\rfloor}$.

\begin{proposition}[\cite{ENV}] 
 The topological twist knot type $K_{-n-1}$ with $n=2k+1\geq 3$ odd has $\mbox{max-tb}(K_{-n-1}) =1$.  Any Legendrian $K_{-n-1}$ knot with maximal $\mathit{tb}$ is Legendrian isotopic to $K_{z_+,z_-}$ for some $0 \leq z_-, z_+ \leq k$.  Moreover, the $K_{z_+,z_-}$ are all distinct except that
\[
K_{z_+,z_-} \cong K_{k-z_+, k-z_-}.
\]
\end{proposition}

The Legendrian contact DGA of $K_{z_+,z_-}$ is as follows:

\begin{itemize}
\item The generators of $\mathcal{A}(K_{z_+,z_-})$ are in bijection with the generators of $\mathcal{A}(\Lambda_n)$.  For any $K_{z_+,z_-}$, let $a$ and $b$ denote the two pictured crossings to the right of $X$; $c_n$ is the crossing to the left of $X$; working clockwise around the knot $c_1, \ldots, c_{n-1}$ are the crossings in $X$ and $e_0, e_1, \ldots, e_n$ are the generators from right cusps.   
\item The grading of generators of $K_{z_+,z_-}$ agrees with the grading in $\Lambda_n$ mod $2$.  (However, the $\mathbb{Z}$-gradings of the DGAs are distinct and can distinguish some of the non-isotopic $K_{z_+,z_-}$ from one another.)
\item If $z_+ < k$ and $z_- < k$, then the {\it abelianized} differentials of all generators are exactly the same as in $\Lambda_n$  (except in the $n=3$ case where the additional $(-b)(1+c_2c_3)$ term in (\ref{eq:2}) does not appear for $K_{0,0}$).
\item Up to Legendrian isotopy, the only max-tb knots not fitting into the above case are $K_{k,0} \cong K_{0,k}$.  For $K_{k,0}$, the abelianized differentials are the same except that (\ref{eq:2}) becomes
\[
\partial e_1 = 1+(1+ba)(-c_1) + (-b)(1+c_{n-1}c_n)(1+c_{n-3}c_{n-2}) \cdots (1+c_2c_3).
\]
\end{itemize}
To verify the computation of differentials, it is useful to keep in mind that for the Ng resolution of a front diagram, except for the standard disks at right cusps, the absolute maximum of the $x$-coordinate of any rigid disk can only occur at a positive puncture;  see \cite[Figure 5]{Ng} and the surrounding discussion. 

\begin{figure}[!ht]
\labellist
\large
\pinlabel $X$  at 122 34
\small
\pinlabel $Z^+$ [t] at 362 32
\pinlabel $Z^-$ [t] at 452 32
\pinlabel $S^+$ [t] at 532 32
\pinlabel $S^-$ [t] at 612 32

\endlabellist
\includegraphics[scale=.6]{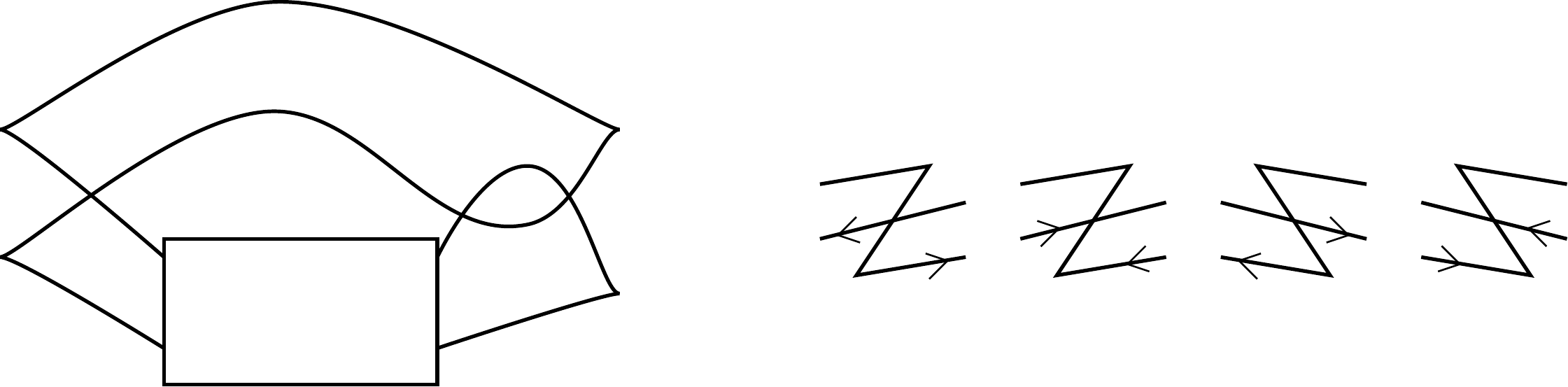}
\caption{Maximal Thurston-Bennequin number twist knots, and the tangles $Z^+,Z^-, S^+, S^-$. }
\label{fig:TwistKnot}
\end{figure}
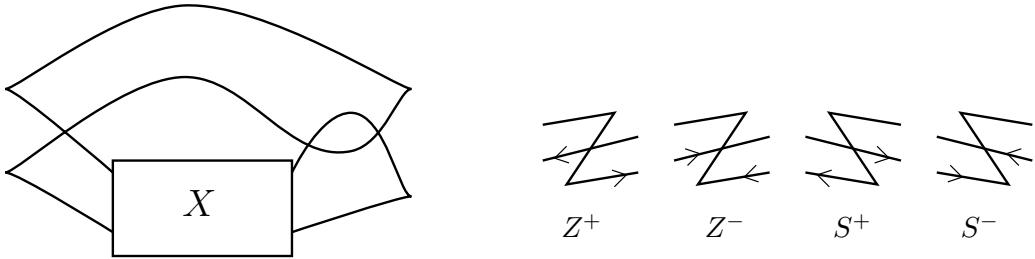

In particular, in all cases the $m=2$ graded augmentation variety has $\mathit{Aug}_{2}(K_{z_+,z_-}, \mathbb{F}) \cong \mathit{Aug}(\Lambda_n, \mathbb{F})$.  Thus, the same argument as above establishes:

\begin{corollary}  For any $0 \leq z_+,z_- \leq k$ with either $z_+,z_- <k$ or $(z_+,z_-) = (k,0)$, $\mathcal{A}(K_{z_+,z_-})$ has a $2$-graded augmentation with $(a,b) = (0,0)$ that cannot be induced by any orientable Lagrangian filling.  In particular, any Legendrian negative twist knot with odd crossing number $\geq 5$ and maximal Thurston-Bennequin number has such an augmentation.
\end{corollary}

Of the three augmentations of these $K_{z_+,z_-}$ to $\mathbb{F}_2$, the other two can be induced by genus $1$ fillings similar to those pictured in Figure \ref{fig:TwistKnot}.

\begin{remark}
Note that the topological invariance of the $2$-graded augmentation variety exhibited here for negative twist knots is a special case of Ng's conjecture about the abelianized $2$-graded characteristic algebra, cf. \cite[Conjecture 3.14]{Ng}.  That the point count of the $2$-graded augmentation varieties over finite fields is topological is established in \cite{HR}.   Recalling that a Lagrangian filling is $2$-graded (i.e. has $m(L)$ even) if and only if it is orientable, it seems reasonable to pose the following question.
\end{remark}

\begin{question}
Does the number of oriented Lagrangian fillings up to exact Lagrangian isotopy of a Legendrian knot $K$ only depend on $\mathit{tb}(K)$ and the topological knot type of $K$?
\end{question}


\begin{thebibliography}{99}

\bibitem{ABL}  B. H. An, Y. Bae, E. Lee, \textsl{Lagrangian fillings for Legendrian links of finite type}, preprint, arXiv:2101.01943.

\bibitem{BHS}  A. Banyaga, D. Hurtubise, P. Spaeth, \textsl{Twisted Morse complexes}, preprint, arXiv:1911.07818.

\bibitem{BC} F. Bourgeois, B. Chantraine, \textsl{Bilinearized Legendrian contact homology and the augmentation category},
J. Symplectic Geom. 12 (2014), no. 3, 553--583.


\bibitem{BST} F. Bourgeois, J. Sabloff, L. Traynor, \textsl{Lagrangian cobordisms via generating families: construction and geography}, Algebr. Geom. Topol. 15 (2015), no. 4, 2439--2477.


\bibitem{CG} R. Casals, H. Gao, \textsl{Infinitely many Lagrangian fillings}, preprint, arXiv:2001.01334.

\bibitem{CN} R. Casals, L. Ng, \textsl{Braid Loops with infinite monodromy on the Legendrian contact DGA}, preprint, arXiv:2101.02318.

\bibitem{Chan}  B. Chantraine, \textsl{Lagrangian concordance of Legendrian knots}, Algebr. Geom. Topol. 10 (2010), no. 1, 63--85.


\bibitem{CDGG1} B. Chantraine, G. Dimitroglou Rizell, P. Ghiggini, R. Golovko,  \textsl{Noncommutative augmentation categories}, Proceedings of the G{\"o}kova Geometry-Topology Conference 2015, 116--150, G{\"o}kova Geometry/Topology Conference (GGT), G{\"o}kova, 2016.


\bibitem{CDGG}  B. Chantraine, G. Dimitroglou Rizell, P. Ghiggini, R. Golovko, \textsl{Floer theory for Lagrangian cobordisms}, J. Differential Geom. 114 (2020), no. 3, 393--465.

\bibitem{CEKSW} G. Civan, J. Etnyre, P. Koprowski, J. Sabloff, A. Walker, \textsl{Product structures for Legendrian contact homology}, Math. Proc. Cambridge Philos. Soc. 150 (2011), no. 2, 291--311.


\bibitem{CNS}  B. Chantraine, L. Ng, S. Sivek, \textsl{Representations, sheaves and Legendrian $(2,m)$-torus links},
J. Lond. Math. Soc. (2) 100 (2019), no. 1, 41--82. 

\bibitem{Ch} Y. Chekanov, \textsl{Differential algebra of Legendrian links}, Invent. Math. 150 (2002), no. 3, 441--483.

\bibitem{CET} J. Conway, J. Etnyre, B. Tosun, \textsl{Symplectic fillings, contact surgeries, and Lagrangian disks}, Int. Math. Res. Not, \textsl{to appear}.   

\bibitem{DR}  G. Dimitroglou Rizell,  \textsl{Lifting pseudo-holomorphic polygons to the symplectisation of $P\times \mathbb{R}$ and applications},
Quantum Topol. 7 (2016), no. 1, 29--105.

\bibitem{Ekh2} T. Ekholm, \textsl{Rational symplectic field theory over $\Z_2$ for exact Lagrangian cobordisms}, J. Eur. Math. Soc. (JEMS) 10 (2008), no. 3, 641--704. 

\bibitem{Ekh}  T. Ekholm, \textsl{Rational SFT, linearized Legendrian contact homology, and Lagrangian Floer cohomology}, Perspectives in analysis, geometry, and topology, 109--145,
Progr. Math., 296, Birkhauser/Springer, New York, 2012.



\bibitem{EHK}  T. Ekholm, K. Honda, T. K\'{a}lm\'{a}n, \textsl{Legendrian knots and exact Lagrangian cobordisms},
J. Eur. Math. Soc. (JEMS) 18 (2016), no. 11, 2627--2689. 


\bibitem{EL}  T. Ekholm, Y. Lekili, \textsl{Duality between Lagrangian and Legendrian invariants}, preprint, arXiv:1701.01284. 

\bibitem{EES05a}   T. Ekholm, J. Etnyre,  M. Sullivan, \textsl{The contact homology of Legendrian submanifolds in $\R^{2n+1}$}, J. Differential Geom. 71 (2005), no. 2, 177--305. 

\bibitem{EES05b}  T. Ekholm, J. Etnyre,  M. Sullivan, \textsl{Orientations in Legendrian contact homology and exact Lagrangian immersions}, Internat. J. Math. 16 (2005), no. 5, 453--532.


\bibitem{EES07} T. Ekholm, T. Etnyre, M. Sullivan,  \textsl{Legendrian contact homology in $P \times \R$}, Trans. Amer. Math. Soc. 359 (2007), no. 7, 3301--3335.

\bibitem{Etgu}  T. Etg\"{u},   \textsl{Nonfillable Legendrian knots in the 3-sphere}, Algebr. Geom. Topol. 18 (2018), no. 2, 1077--1088. 

\bibitem{EN} J. Etnyre, L. Ng, \textsl{Legendrian contact homology in $\mathbb{R}^3$}, preprint, arXiv:1811.10966.

\bibitem{ENS} J. Etnyre, L. Ng, J. Sabloff,
\textsl{Invariants of Legendrian knots and coherent orientations},  
J. Symplectic Geom. 1 (2002), no. 2, 321--367.

\bibitem{ENV}  J. Etnyre, L. Ng, V. V\'{e}rtesi, \textsl{Legendrian and transverse twist knots}, J. Eur. Math. Soc. (JEMS) 15 (2013), no. 3, 969--995.

\bibitem{FHT}  Y. F\'{e}lix, S. Halperin,  J-C. Thomas, \textsl{Rational homotopy theory}, Graduate Texts in Mathematics, 205, Springer-Verlag, New York, 2001.


\bibitem{GSW}  H. Gao, L. Shen, D. Weng, \textsl{Augmentations, Fillings, and Clusters}, preprint, arXiv:2008.10793.

\bibitem{Gui} S. Guillermou, \textsl{Quantization of conic Lagrangian submanifolds of cotangent bundles}, preprint, arXiv:1212.5818.

\bibitem{HSab} K. Hayden, J. Sabloff, \textsl{Positive knots and Lagrangian fillability}, Proc. Amer. Math. Soc. 143 (2015), no. 4, 1813--1821. 


\bibitem{HR}  M. Henry, D. Rutherford, \textsl{Ruling polynomials and augmentations over finite fields}, J. Topol. 8 (2015), no. 1, 1--37. 

\bibitem{Hughes}  J. Hughes, \textsl{Weave Realizability for D-type}, preprint, arXiv:2101.10306.

\bibitem{JT} X. Jin and D. Treumann, \textsl{Brane structures in microlocal sheaf theory}, preprint, arXiv:1704.04291

\bibitem{K} C. Karlsson, \textsl{A note on coherent orientations for exact Lagrangian cobordisms}, Quantum Topol. 11 (2020), no. 1, 1--54.


\bibitem{LeRu}  C. Leverson, D. Rutherford, \textsl{Satellite ruling polynomials, DGA representations, and the colored HOMFLY-PT polynomial,}  Quantum Topol. 11 (2020), no. 1, 55--118.


\bibitem{LSab} E. Lipman, J. Sabloff, \textsl{Lagrangian fillings of Legendrian 4-plat knots}, Geom. Dedicata 198 (2019), 35--55.


\bibitem{Mish} K. Mishachev, \textsl{The N-copy of a topologically trivial Legendrian knot},
J. Symplectic Geom. 1 (2003), no. 4, 659--682.


\bibitem{NadlerZ}  D. Nadler and E. Zaslow, \textsl{Constructible sheaves and the Fukaya category},
J. Amer. Math. Soc. 22 (2009), no. 1, 233--286.


\bibitem{Ng}   L. Ng, \textsl{Computable Legendrian invariants}, Topology 42 (2003), no. 1, 55--82.

\bibitem{NRSSZ} L. Ng, D. Rutherford, V. Shende, S. Sivek  E. Zaslow, \textsl{Augmentations are Sheaves}, 
 Geom. Topol. 24 (2020), no. 5, 2149--2286. 

\bibitem{Pan2}  Y. Pan, \textsl{Exact Lagrangian fillings of Legendrian (2,n) torus links}, Pacific J. Math. 289 (2017), no. 2, 417--441.

\bibitem{PanRu2} Y. Pan, D. Rutherford, \textsl{Augmentations and immersed Lagrangian fillings}, preprint, arXiv:2006.16436.

	
\bibitem{Sab}  J. Sabloff, \textsl{Duality for Legendrian contact homology}, Geom. Topol. 10 (2006), 2351--2381.


	
\bibitem{Sei} P. Seidel. \textsl{Fukaya categories and Picard-Lefschetz theory}, Zurich Lectures in Advanced Mathematics. European Mathematical Society (EMS), Zurich, 2008. 


\bibitem{STWZ} V. Shende, D. Treumann, H. Williams, E. Zaslow, 
\textsl{Cluster varieties from Legendrian knots}, 
Duke Math. J. 168 (2019), no. 15, 2801--2871.

\bibitem{STZ}
V. Shende, D. Treumann,  E. Zaslow, \textsl{Legendrian knots and constructible sheaves}, Invent. Math. 207 (2017), no. 3, 1031--1133. 



\end{thebibliography}
\end{document}